\documentclass
[
    a4paper,
    DIV=11,
]
{scrartcl}
\usepackage{fullpage,authblk,amsmath,amssymb,amsthm,amsfonts}
\usepackage{bm, bbm, dsfont}
\usepackage{enumerate}
\usepackage{tikz}
\usepackage{caption}
\usepackage{subcaption}
\usepackage{url,mathtools}
\usepackage{tensor}

\usepackage[hidelinks]{hyperref}

\usepackage{cleveref}

\newtheorem{theorem}{Theorem}[section]

\newtheorem{lemma}[theorem]{Lemma}
\newtheorem{proposition}[theorem]{Proposition}

\newtheorem{assumption}[theorem]{Assumption}

\theoremstyle{definition}
\newtheorem{remark}[theorem]{Remark}

\numberwithin{equation}{section}

\newcommand{\ud}{\,\mathrm{d}}
\newcommand{\bb}[1]{\mathbb{#1}}

\newcommand{\abs}[1]{\vert #1 \vert}
\newcommand{\ang}[1]{\left\langle #1 \right\rangle}
\newcommand{\vrt}[1]{\left\Vert #1 \right\Vert}
\newcommand{\Abs}[1]{\left\vert#1\right\vert}
\newcommand{\Ang}[1]{\left\langle #1 \right\rangle}
\newcommand{\Vrt}[1]{\left\Vert #1 \right\Vert}

\newcommand*{\email}[1]{\bgroup\href{mailto:#1}{#1}\egroup}

\begin{document}
\title{Randomised one-step time integration methods for deterministic operator differential equations}

\author{%
	Han\ Cheng\ Lie\footnote{~Institut f\"ur Mathematik, Universit\"at Potsdam, Campus Golm, Haus 9, Karl-Liebknecht-Stra{\ss}e 24--25, Potsdam OT Golm 14476, Germany, \email{hanlie@uni-potsdam.de}}
	\quad
	Martin\ Stahn\footnote{~Institut f\"ur Mathematik, Universit\"at Potsdam, Campus Golm, Haus 9, Karl-Liebknecht-Stra{\ss}e 24--25, Potsdam OT Golm 14476, Germany, \email{martin.stahn@uni-potsdam.de}} 
	\quad
	T.\ J.\ Sullivan\footnote{~Mathematics Institute and School of Engineering, University of Warwick, Coventry, CV4 7AL, United Kingdom \email{t.j.sullivan@warwick.ac.uk}} \footnote{~ Alan Turing Institute, 96 Euston Road, London, NW1~2DB, United Kingdom}
}

\date{}
\maketitle

\begin{abstract}
    Uncertainty quantification plays an important role in problems that involve inferring a parameter of an initial value problem from observations of the solution.
    Conrad et al.\ (\textit{Stat.\ Comput.}, 2017) proposed randomisation of deterministic time integration methods as a strategy for quantifying uncertainty due to the unknown time discretisation error.
    We consider this strategy for systems that are described by deterministic, possibly time-dependent operator differential equations defined on a Banach space or a Gelfand triple. 
    Our main results are strong error bounds on the random trajectories measured in Orlicz norms, proven under a weaker assumption on the local truncation error of the underlying deterministic time integration method. 
    Our analysis establishes the theoretical validity of randomised time integration for differential equations in infinite-dimensional settings.
\end{abstract}






This version of the article has been accepted for publication, after peer review (when applicable) but is not the Version of Record and does not reflect post-acceptance improvements, or any corrections. The Version of Record is available online at the link below.
\begin{center}
 \url{http://dx.doi.org/10.1007/s10092-022-00457-6}
\end{center}

\section{Introduction}
\label{sec:intro}

The numerical solution of deterministic dynamical systems is an important task in many applications where the dynamical system is a spatiotemporal field that satisfies a partial differential equation (PDE). In this case, the field can be viewed as a function $u$ mapping to an infinite-dimensional real separable Banach space $(V,\Abs{\cdot}_V)$, and the dynamical system is described by a deterministic operator differential equation initial value problem on a finite time interval $[0,T]$ for some initial condition $\vartheta$:
\begin{equation*}
    u(0)=\vartheta,\quad u'(t)=f(t,u(t)), \quad t\in [0,T].
\end{equation*}
Operator differential equations have been applied in peridynamics and elastic materials, e.g.\ \cite{Emmrich2007,Meinlschmidt2020}. 
The purpose of this paper is to analyse the error of randomised time integration methods for solving such initial value problems. 
The methods are of the form
\begin{equation*}
U_{k+1}\coloneqq \psi(h,U_{k})+\xi_k(h),\quad k\in\{0,\ldots,N-1\},
\end{equation*}
where $\psi(h,U_{k})$ represents the output of a deterministic time integration method with time step $h$ corresponding to the input $U_{k}$, and $\xi_k(h)$ is a $V$-valued random variable whose distribution depends on $h$.
Our motivation for considering these methods comes from Bayesian inverse problems. 

In many applications, the initial value problem depends on a parameter $\theta^\ast$ --- for example, the initial condition $\vartheta$, or a parameter appearing in the vector field $f$ --- and it is of interest to infer the value of $\theta^\ast$ given some observational data $y$, where $y$ results from some fixed measurement process. Let $\Theta$ and $\mathcal{Y}$ denote the set of feasible parameter values and the set of feasible data values respectively.
We assume that $\Theta$ is a Banach space and $\mathcal{Y}$ is a Hilbert space. Let $S$ denote the solution operator that maps every $\theta'\in\Theta$ to the solution of the corresponding initial value problem, and let $O$ denote the observation operator that maps every continuous trajectory in $V$ to the corresponding output $\tilde{y}\in\mathcal{Y}$ of the fixed measurement process.
Then the inference problem is to determine the value of the unknown true parameter $\theta^\ast$ given noisy data of the form
\begin{equation*}
 y=O\circ S(\theta^\ast)+\eta,
\end{equation*}
where $\eta$ is often assumed to be a centred Gaussian random variable with known,  positive-definite covariance operator $\Gamma$. In general, the inverse problem is ill-posed, and one can apply deterministic or statistical approaches to solving the inverse problem.

In the Bayesian approach to inverse problems, one assumes that $\Theta$ can be equipped with a probability measure $\mu_{0}$, called the `prior'. Let $G\coloneqq O\circ S:\Theta \to\mathcal{Y}$ denote the parameter-to-observable map. The Bayesian solution to the inverse problem is given by the `posterior' probability measure $\mu^{y}$ on $\Theta$, which satisfies
\begin{equation*}
 \mu^{y}(\ud \theta')=\frac{1}{Z(y)} \exp\left(-\frac{1}{2}\Vrt{y-G(\theta')}_{\Gamma}^{2}\right)\mu_{0}(\ud \theta')
\end{equation*}
where $\vrt{x}_{\Gamma}^{2}=\ang{x,\Gamma^{-1} x}_{\mathcal{Y}}$ and $Z(y)$ is a normalisation constant.
The posterior is important because one can use it to perform uncertainty quantification for the unknown parameter $\theta^\ast$. See \cite[Section 2.4]{Stuart2010} for a presentation of the Bayesian approach to inverse problems posed on vector spaces.

For many differential equations arising in applications, one must approximate the exact solution operator $S$ using another operator $\tilde{S}$ that results from a discretisation of the initial value problem. 
This leads to an approximation $\tilde{G}\coloneqq O\circ\tilde{S}$ of the parameter-to-observable map, which in turn leads to an approximation $\tilde{\mu}^{y}$ of the exact posterior $\mu^{y}$ defined above.
For a fixed data vector $y$ and prior $\mu_{0}$, the error in $\tilde{S}$ is propagated via Bayes' theorem to an error in $\tilde{\mu}^{y}$. 
Since the posterior is fundamental for performing inference on the unknown parameter $\theta^\ast$, one seeks a principled way to take into account the discretisation error in $\tilde{S}$.

Under some assumptions, a bound on the error $G-\tilde{G}$ with respect to some appropriate norm can be used to prove a bound on the error in the posterior, as measured by the Hellinger metric, e.g.\ \cite[Corollary 4.9]{Stuart2010}. 
Stability bounds of this type ensure that the approximate posterior $\tilde{\mu}^{y}$ converges in the Hellinger metric to the exact posterior $\mu^{y}$, in the limit as the discretisation error vanishes. 
While this property ensures that we can ignore the error in the posterior in the limit of increasingly finer discretisations, it does not indicate how to treat the error in the posterior for a fixed discretisation.

One approach is to ignore the discretisation error. 
This approach is not ideal from the point of view of statistical inference, because the approximate posterior $\tilde{\mu}^{y}$ can be tightly concentrated around the wrong parameter values, even in the small-noise limit.
This phenomenon of `overconfidence' is undesirable for uncertainty quantification.
See Section \ref{ssec:example_overconfidence} below.

The approach presented in \cite{Conrad2017} approaches the problem of accounting for the discretisation error, by applying the standard procedure of using random variables as proxies for unknown quantities. 
Let $\psi(h,v)$ denote the output of applying a time integration method for time step $h$ to the state $v$, for a fixed time step $h=T/N>0$, $N\in\bb{N}$. 
Consider the error $u(h)-\psi(h,u(0))$ between the exact solution and the numerical solution, incurred over one time step. Since the one-step error is unknown, we model it using a random variable $\xi_0(h)$. Thus, 
\begin{equation*}
 u(h)\approx \psi(h,u(0))+\xi_0(h) \eqqcolon U_1.
\end{equation*}
If we model the one-step error for subsequent steps in a similar way, then this leads to the randomised time integration methods stated at the beginning of this section.

\subsection{Illustration of overconfidence phenomenon}
\label{ssec:example_overconfidence}

Consider the standard heat equation on a bounded domain $D\subset\bb{R}^{d}$ with homogeneous Dirichlet boundary conditions, written as the operator differential equation
\begin{equation*}
 u(0)=\vartheta\in H,\quad u'(t)+A u(t)=0,\quad t\in [0,h],
\end{equation*} 
where $A$ is the Laplacian, $H=L^2(D)$, and $h>0$.
In \cite[Section 3.5]{Stuart2010}, one considers the inverse problem of inferring the initial condition $\vartheta$ from a noisy observation of the solution at a later time.
We shall use the assumptions and the approach stated there.
The parameter-to-observable map is $G:H\to H$, $v\mapsto e^{-hA}v$.
The data $y$ is a realisation of the random variable
\begin{equation*}
 Y=G(\vartheta)+\delta^{1/2}\eta=G\vartheta +\delta^{1/2}\eta
\end{equation*}
where the noise $\eta$ is a Gaussian random variable with distribution $\mathcal{N}(0,\Gamma_{\textup{obs}})$.
The noise scaling $\delta$ is assumed to be known, and the small noise limit corresponds to $\delta\to 0$. 
For the unknown parameter $\vartheta$, we use the Gaussian prior $\mu_0=\mathcal{N}(m_0,\Gamma_{0})$.
The positive-definite covariance operators $\Gamma_\textup{obs}$ and $\Gamma_{0}$ are chosen so that 1) draws from $\mathcal{N}(0,\Gamma_{\textup{obs}})$ and from $\mu_{0}$ are $H$-valued, almost surely; and 2) $\Gamma_{0}$ is a power of $A$.
Applying \cite[Theorem 6.20]{Stuart2010} to the jointly Gaussian random variable $(U,G(U)+\delta^{1/2}\eta)$ with $U\sim\mu_0$ yields the Gaussian posterior measure $\mu^{y}$ with mean and covariance
\begin{align*}
 m=&m_0+\Gamma_{0} G (\delta \Gamma_{\textup{obs}}+G\Gamma_{0} G )^{-1} (y-Gm_0)
 \\
 \mathcal{C}=&\Gamma_{0}-\Gamma_{0}G (\delta \Gamma_{\textup{obs}}+G\Gamma_{0} G )^{-1}G\Gamma_{0}.
\end{align*}
In the $\delta \to 0$ limit, $y\to G\vartheta$. 
Using this fact and the assumptions on $\Gamma_{0}$, it follows that $\mathcal{C}\to 0$ and $ m\to\vartheta$ in the $\delta\to 0$ limit.
Since Gaussian measures are completely characterised by their mean and covariance, the convergence of $\mathcal{C}$ and $m$ implies the weak convergence (in the sense of probability measures) of the posterior measure to the Dirac measure at the true initial condition $\vartheta$ as $\delta \to 0$. 
This convergence captures the concentration of the posterior $\mu^y$ around the true unknown $\vartheta$, and validates the Bayesian approach to the inverse problem.

Now suppose we approximate $G$ using the map $\tilde{G}$ defined by the implicit Euler method, $\tilde{G}:H\to H$, $v\mapsto (I+hA)^{-1} v$. 
Applying \cite[Theorem 6.20]{Stuart2010} as we did earlier with $\tilde{G}$ instead of $G$ yields the associated approximate posterior $\tilde{\mu}^{y}$, which is Gaussian with mean and covariance
\begin{align*}
 \tilde{m}=&m_0+\Gamma_{0} \tilde{G} (\delta \Gamma_{\textup{obs}}+\tilde{G}\Gamma_{0} \tilde{G} )^{-1} (y-\tilde{G}m_0)
 \\
\tilde{ \mathcal{C}}=&\Gamma_{0}-\Gamma_{0}\tilde{G} (\delta \Gamma_{\textup{obs}}+\tilde{G}\Gamma_{0} \tilde{G} )^{-1}\tilde{G}\Gamma_{0}.
\end{align*}
In the $\delta\to 0$ limit, $\tilde{\mathcal{C}}\to 0$, but $\tilde{m}\to\tilde{G}^{-1}G \vartheta\neq \vartheta$.
Thus, the approximate posterior $\tilde{\mu}^{y}$ converges weakly in the small noise limit to a biased Dirac measure. 
This demonstrates the overconfidence phenomenon. 
The bias $\tilde{G}^{-1}G\vartheta-\vartheta$ in the limiting Dirac measure is the local truncation error of the implicit Euler method. 

To address the overconfidence phenomenon, we use a random variable as a proxy for the unknown bias. 
Consider the randomised implicit Euler method given by $\widehat{G}(v)\coloneqq \tilde{G}v+h^{p+1}\zeta $, where $\zeta\sim \mathcal{N}(0,\Gamma_{1})$ is independent of the observation noise $\eta$, and $\Gamma_{1}$ is chosen so that draws from $\mathcal{N}(0,\Gamma_{1})$ are $H$-valued almost surely. 
By rewriting $\widehat{G}(U)+\delta^{1/2}\eta=\tilde{G}U+(h^{p+1}\zeta+\delta^{1/2}\eta)$ and applying \cite[Theorem 6.20]{Stuart2010}, it follows that the associated deterministic posterior $\widehat{\mu}^{y}$ is Gaussian, with mean and covariance
\begin{align*}
 \widehat{m}=&m_0+\Gamma_{0} \tilde{G} (\delta \Gamma_{\textup{obs}}+h^{2p+2}\Gamma_{1}+\tilde{G}\Gamma_{0} \tilde{G} )^{-1} (y-\tilde{G}m_0)
 \\
\widehat{ \mathcal{C}}=&\Gamma_{0}-\Gamma_{0}\tilde{G} (\delta \Gamma_{\textup{obs}}+h^{2p+2}\Gamma_{1}+\tilde{G}\Gamma_{0} \tilde{G} )^{-1}\tilde{G}\Gamma_{0}.
\end{align*}
In the $\delta\to 0$ limit, $\widehat{C}$ does not converge to zero, because of the additional $h^{2p+2}\Gamma_{1}$ term. 
This term ensures that the deterministic approximate posterior $\widehat{\mu}^{y}$ associated to the randomised implicit Euler method $\widehat{G}$ is more `spread out' than the approximate posterior $\tilde{\mu}^{y}$ associated to the non-randomised implicit Euler method $\tilde{G}$.
In this way, the problem of overconfidence is mitigated.

\subsection{Main contributions} 

In this paper, we rigorously prove strong forward error bounds for randomised one-step time integration methods applied to operator differential equations. 
Our work builds on the approach for proving the error bounds in $L^2$ of \cite[Theorem 2.2]{Conrad2017} and the error bounds in $L^R$ --- for user-specified $R\in\bb{N}$ --- of \cite[Theorem 3.5]{LieStuartSullivan2019}. 
These bounds were stated for initial value problems formulated in $\bb{R}^{d}$, where the associated exact flow maps are globally Lipschitz, and where the randomised time integrators are generated using uniform time grids and numerical methods $\psi$ that satisfy a uniform local truncation error assumption. 

The error bounds that we prove in this paper generalise the existing error bounds in multiple aspects. 
Our bounds are valid for time-dependent vector fields, non-uniform time grids (i.e.\ variable time steps), and operator differential equations that are formulated on Banach spaces or on Gelfand triples. 
In Theorem \ref{thm:LR_bound_max_error_Gelfand}, we show that one can obtain strong error bounds in $L^R$ for $R>1$, without the assumption of uniform local truncation error of the numerical method, and without the assumption that the flow map of the initial value problem is globally Lipschitz. 
In fact, we show that one can obtain strong error bounds in more general Orlicz norms.
The bounds that we prove in this paper demonstrate that the paradigm of randomised time integration extends in a natural way to the time integration for PDEs with time-dependent coefficients. 
Moreover, the proofs we give for our main results are simpler than the proofs of the corresponding results given in \cite{LieStuartSullivan2019}.

A related but distinct contribution that we make is to consider the setting where the random variables used in the randomisation are independent and centred. We generalise the $L^2$ uniform error bound \cite[Theorem 3.4]{LieStuartSullivan2019} for centred and independent randomisation --- which was proven in the setting of ODEs in $\bb{R}^{d}$ --- to the setting of operator differential equations on Gelfand triples, under weaker assumptions on the time integration map $\psi$. We address the question of whether it is possible to obtain better error bounds under these additional assumptions. This question was implicit in the analysis of \cite{LieStuartSullivan2019}, but was not addressed there. 

\subsection{Related work} 

Randomised time integration methods for differential equations have been studied extensively in the context of `probabilistic numerics'. For some reviews of research in this area, see \cite{HennigOsborneGirolami2015,Cockayne2019,Oates2019}. In probabilistic numerics, ODEs have been considered from many perspectives, including structure- or symmetry-preserving methods \cite{AbdulleGaregnani2020,Wang2020}, Bayesian modelling of the unknown solution with Gaussian processes \cite{Teymur2016,Conrad2017,Chkrebtii2019,Schober2019,Tronarp2019,Wang2020}, data-based statistical estimation of discretisation error \cite{Matsuda2019,Teymur2018}, and filtering \cite{Kersting2020,Tronarp2019}. The papers \cite{Conrad2017,LieStuartSullivan2019} cited earlier also belong to this context. For PDEs, methods based on Bayesian inference and Gaussian processes \cite{Owhadi2015,Chkrebtii2016,Conrad2017,Cockayne2017a,Raissi2018,Wang2021}, multiscale techniques \cite{Owhadi2017}, and random meshes \cite{AbdulleGaregnani2021} have been studied. The research area of `information field dynamics' \cite{Ensslin2013,DupontEnsslin2018a} also considers probabilistic simulation schemes for PDEs by using Gaussian processes and information theoretic ideas. 

Random approximate posteriors arising from randomised solution operators for differential equations have been studied in \cite[Section 5] {LieSullivanTeckentrup2018} under a strong assumption of exponentially integrable discretisation error $S-\tilde{S}$, and more recently under a weaker square integrability hypothesis in \cite{Garegnani2021}.

Two aspects differentiate the problem we consider from the problems considered in numerical methods for stochastic evolution equations. The most important aspect is that the operator differential equation of interest in this paper is deterministic. Thus, our context is fundamentally different from the context of numerical integration methods for stochastic differential equations and numerical integration methods for random differential equations. The second aspect is that the random variables used in the randomisation need not be constructed using i.i.d. copies of a Wiener process or L\'evy process.

\subsection{Overview}

We introduce notation and some recurring objects in the next section.
In Section \ref{sec:Banach}, we consider the setting where the initial value problem is formulated on a Banach space.
The main result is the strong error bound in Orlicz norm proven in Theorem \ref{thm:LR_bound_max_error_Banach} under the assumption of uniform local truncation error of the time integration method $\psi$.

In Section \ref{sec:Gelfand}, we consider the setting where the initial value problem is formulated on a Gelfand triple, and where $\psi$ satisfies a weaker local truncation error assumption.
This setting is considered in the variational approach to PDEs.
We prove strong $L^2$ error bounds for mutually independent and centred randomisation in Section \ref{ssec:L2_error_bounds_independent_centred_case_Gelfand}.
 In Section \ref{ssec:higher_order_error_bounds_independent_centred_Gelfand}, we discuss the feasibility of obtaining $L^R$ bounds for $R>2$ that are of the same order in the time step $h$, under the same assumptions of independence and centredness. 
In Section \ref{ssec:higher_order_error_bounds_Gelfand}, we state in Theorem \ref{thm:LR_bound_max_error_Gelfand} a strong error bound in Orlicz norm without assuming independence or centredness. 

In Section \ref{sec:example}, we show that the assumptions we make in Section \ref{sec:Gelfand} are reasonable for a class of operator differential equations that includes the heat equation on a $C^2$ bounded domain.

We conclude in Section \ref{sec:conclusion}. In the appendices, we collect material that is useful for the main part of the paper.

\subsection{Notation and setup}

Below, $(V,\Abs{\cdot}_V)$ and $(H,\ang{\cdot,\cdot}_H)$ denote a real separable Banach space and a real separable Hilbert space respectively. 
We write $\Abs{\cdot}_H$ for the Hilbert space norm.
All integrals are Bochner integrals unless otherwise stated.
We define $C^1([0,T];V) \coloneqq \{ u \in C([0,T];V) \, | \, u' \in C([0,T];V)\}$ and equip it with the norm $\Vrt{u }_{1,\infty} = \Vrt{  u}_{\infty} + \Vrt{ u'}_{\infty}$ where $\Vrt{\cdot }_{\infty}$ denotes the supremum norm on $[0,T]$.
We define the space $C^1([0,T]; H)$ analogously.

All random variables will be defined on a common probability space $(\Omega,\mathcal{F},\bb{P})$.
We denote expectation with respect to $\bb{P}$ by $\bb{E}[\cdot]$ and write $X\sim \mu$ to mean that $X$ has $\mu$ as its distribution.
For a $V$-valued random variable $X$ and $R\geq 1$, we shall write $\vrt{X}_{L^R(\Omega;V)}\coloneqq \bb{E}[\abs{X}_V^R]^{1/R}$.
Similarly, if $X$ is $H$-valued, then $\vrt{X}_{L^R(\Omega;H)}\coloneqq \bb{E}[\Abs{X}_H^R]^{1/R}$. 
For a Young function $\Psi:\bb{R}_{\geq 0}\to\bb{R}_{\geq 0}$, the corresponding Orlicz norm\footnote{See \cite[Chapter 8]{Adams2004} for a general introduction to Orlicz spaces and norms.} $\vrt{\cdot}_{\Psi}$ of a $\bb{R}$-valued random variable $Z$ is defined by
\begin{equation*}
 \Vrt{Z}_{\Psi}\coloneqq \inf\{k\in (0,\infty)\ :\ \bb{E}[\Psi_q(\abs{Z}/k)]\leq 1\}.
\end{equation*}
If $Z$ is a $V$-valued (respectively, $H$-valued) random variable, then $\Vrt{Z}_{\Psi(\Omega;V)}\coloneqq\Vrt{\abs{Z}_V}_{\Psi}$ (resp. $\Vrt{Z}_{\Psi(\Omega;H)}\coloneqq\Vrt{\abs{Z}_H}_{\Psi}$).
The $\vrt{\cdot}_{\Psi(\Omega;V)}$ norm includes as a special case the $\vrt{\cdot}_{L^R(\Omega;V)}$ norm when $R>1$, but not when $R=1$.
The analogous statement holds for the $\vrt{\cdot}_{\Psi(\Omega;H)}$ norm.
An important choice of Young function $\Psi$ is given by $\Psi_2(z)\coloneqq \exp(z^2)-1$, because finiteness of $\vrt{X}_{\Psi_2}$ implies that $X$ is sub-Gaussian and hence exponentially square integrable. 

We write $p\wedge q=\min\{p,q\}$ for $p,q\in\bb{R}$. For $h>0$, $p\geq 0$, and $a=a(h)\in\bb{R}$, we write $a=\mathcal{O}(h^p)$ to mean that $\abs{a}\leq Ch^p$ for some $h$-independent term $C>0$.

Given $N\in\bb{N}$, $[N]\coloneqq \{1,\ldots,N\}$ and $[N]_0\coloneqq [N]\cup\{0\}=\{0,1,\ldots, N\}$.

Throughout the paper, we consider the following initial value problem on a deterministic time interval $[0,T]$,
\begin{equation}
    \label{eq:operator_differential_equation}
    u(0)=\vartheta,\quad u'(t)=f(t,u(t)),\quad t\in [0,T]
\end{equation}
for fixed $T>0$ and suitable initial condition $\vartheta$. 
We specify the domain and codomain of $f$ in the following sections.
We denote by $\varphi$ the exact flow map associated to \eqref{eq:operator_differential_equation} as follows: for suitable $h\in[0,T]$, $t\in[0,T-h]$, and $u_s$, 
\begin{equation}
 \label{eq:exact_flow_map}
    \varphi(h,t,u_s)=u_s+\int_{t}^{t+h}f(\tau,\varphi(\tau,t,u_s))\ud \tau.
\end{equation}
We equip the time interval $[0,T]$ in \eqref{eq:operator_differential_equation} with a time grid $(t_k)_{k\in [N]_0}$, where
\begin{equation}
    \label{eq:time_grid}
    0\eqqcolon t_0<t_1<\cdots<t_N\coloneqq T,\quad h_{k}\coloneqq t_{k+1}-t_k,\quad h\coloneqq \max_{k\in [N-1]_0}h_k.
\end{equation}
From \eqref{eq:time_grid} it follows that for any $\tau\geq 0$,
\begin{equation}
\label{eq:upper_bound_on_sum_time_steps_to_power_plus_one}
 \sum_{\ell\in[N-1]_0}h_\ell^{\tau+1}\leq h^\tau\sum_{\ell\in[N-1]_0}h_\ell=h^\tau T.
\end{equation}
Given \eqref{eq:exact_flow_map}, the exact sequence $(u(t_k))_{k\in[N]_0}$ associated with the time grid satisfies
\begin{equation}
\label{eq:det_exact_seq}
 u(t_{k+1})=\varphi(h_k,t_k,u(t_k)),\quad k\in [N-1]_0.
\end{equation}
We denote by $\psi$ the approximate flow map associated to a time integration method, and define a deterministic approximating sequence $(u_k)_{k\in [N]_0}$ by
\begin{equation*}
	u_{k+1} \coloneqq \psi(h_k, t_k,u_k),\quad u_0=\vartheta.
\end{equation*}
Let $(\xi_k)_{k\in\bb{N}_0}$ be a sequence of stochastic processes, where each $\xi_k$ is a stochastic process on $[0,\infty)$. 
In Section \ref{sec:Banach} (respectively, Section \ref{sec:Gelfand}), each $\xi_k$ takes values in the Banach space $V$ (resp. the Hilbert space $H$). 
Given the time grid in \eqref{eq:time_grid}, we use $(\xi_k(h_k))_{k\in[N-1]_0}$ as a randomisation sequence in order to define the random approximating sequence $(U_k)_{k\in[N]_0}$ by
\begin{equation}
\label{eq:rand_approx_seq}
	U_{k+1} \coloneqq \psi (h_k,t_k,U_k) + \xi_k(h_k),\quad k\in[N-1]_0
\end{equation}
for a given random variable $U_0$. 
The sequence of errors $(e_k)_{k\in[N]_0}$ of the random approximating sequence \eqref{eq:rand_approx_seq} with respect to the exact sequence \eqref{eq:det_exact_seq} is defined by
\begin{equation*}
 e_0=u(0)-U_0,\quad e_{k+1}\coloneqq u(t_{k+1})-U_{k+1},\quad k\in[N-1]_0.
\end{equation*}
By \eqref{eq:det_exact_seq} and \eqref{eq:rand_approx_seq}, we obtain
\begin{equation}
 \label{eq:rand_error_seq_v2}
 e_{k+1}=\varphi(h_k,t_k,u(t_k))-\psi(h_k,t_k,U_k)-\xi_k(h_k),\quad k\in[N-1]_0.
\end{equation}
The equation \eqref{eq:rand_error_seq_v2} shall be the starting point for our error analysis.

\section{Classical setting}
\label{sec:Banach}

In this section, we prove the generalisation of \cite[Theorem 2.2]{Conrad2017} and \cite[Theorem 3.5]{LieStuartSullivan2019} to the setting of a time-dependent vector field $f$ on an infinite-dimensional, real, separable Banach space $V$.
We assume that the vector field $f$ in \eqref{eq:operator_differential_equation} satisfies $f \colon [0,T] \times V \rightarrow V$. In addition, we assume that for every initial condition $\vartheta \in V$, there exists a unique classical solution $u \in C^1([0, T];V)$. 
For example, if $f$ is continuous and uniformly Lipschitz in the second argument, then this assumption is satisfied, and $\varphi$ exists \cite[Satz 7.2.6]{Emmrich2004}.

We state the assumptions needed to prove the main result of this section. 
The first is a Lipschitz continuity assumption on the exact flow map.
\begin{assumption}
    \label{asmp:Lipschitz_flow_map_Banach}
    The exact flow map $\varphi$ admits a constant $L_\varphi>0$ such that for any $t\in [0,T]$, for every $h\geq 0$ such that $t+h\leq T$, and for every $x,y\in V$,
    \begin{equation*}
        \Abs{\varphi(h,t,x)-\varphi(h,t,y)}_V  \leq (1+L_\varphi h) \Abs{x-y}_V.
    \end{equation*}
\end{assumption}
If $f$ is uniformly Lipschitz in the second argument, then Assumption \ref{asmp:Lipschitz_flow_map_Banach} is satisfied \cite[Satz 7.3.4]{Emmrich2004}. 

Ideally, the deterministic sequence $(u_k)_k$ approximates the exact sequence $(u(t_k))_k$ well.
We make this precise by introducing the following uniform local truncation error assumption. 
\begin{assumption}
    \label{asmp:Uniform_truncation_error_Banach}
    The approximate flow map $\psi$ admits constants $0<h^\ast<\infty$, $0<C_{\varphi,\psi}<\infty$, and $q\geq 0$, such that for all $0<h\leq h^\ast$,
    \begin{equation*}
	    \sup_{\substack{v \in V\\ t \in [0,T-h]}} \Abs{ \varphi(h,t,v) - \psi(h,t,v) }_V \leq C_{\varphi,\psi} h^{q+1} \; .
    \end{equation*}
\end{assumption}
The parameter $h^\ast$ is included in order to account for implicit time integration methods that provide a unique output whenever the time step is small enough.
In order to achieve an order of $q\geq 1$ for the truncation error, one usually requires higher regularity of $f$ or equivalently higher regularity for the solution $u$ \cite[Section III.2, Theorem 2.4]{Hairer1993}. 
For classical one-step methods, the corresponding analysis extends to infinite-dimensional Banach spaces; see Appendix \ref{app:one-step}.

The assumptions above are similar to \cite[Assumption 2]{Conrad2017} and \cite[Assumption 3.1, 3.2]{LieStuartSullivan2019}. Note that Assumption \ref{asmp:Uniform_truncation_error_Banach} is restrictive, because it requires uniformity in $t$ and $v$. 
For example, in \cite{Conrad2017}, the analogous assumption is justified under the assumption that $f \colon \bb{R}^{d}\to\bb{R}^{d}$ is sufficiently smooth and sufficiently many of its derivatives are uniformly bounded.
However, Assumption \ref{asmp:Uniform_truncation_error_Banach} is not satisfied in general. 
For example, in the setting where the operator differential equation is given by $u'(t)=Au(t)\in H$ for a Hilbert space $H$ and the infinitesimal generator $A$ of an analytic semigroup with domain $\textup{Dom}(A)$, and $\psi$ is given by the implicit Euler method, there exists $C>0$ such that for all $\vartheta\in\textup{Dom}(A)$, $n\in\bb{N}$, and all sufficiently small $h>0$,
\begin{equation*}
 \Abs{\varphi(nh,0,\vartheta)-\psi(nh,0,\vartheta)}_H\leq 
 C h\Abs{A\vartheta}_H,
\end{equation*}
see \cite[Theorem 7.1]{Thomee2006}.

For equations of the form \eqref{eq:operator_differential_equation} derived from PDEs and fixed time argument $t$, the right hand side $f$ is in many cases not Lipschitz from $V$ to $V$. 
Furthermore, one cannot in general expect that \eqref{eq:operator_differential_equation} admits a classical solution $u\in C^1([0,T];V)$, because a classical solution requires regularity assumptions on the problem data that need not hold in general.
In Section \ref{sec:Gelfand}, we will consider vector fields $f$ that do not satisfy the assumptions above. This will lead us to consider variational solutions of \eqref{eq:operator_differential_equation}.
There are other approaches to generalise the classical setting to problems with less regularity, e.g.\ mild solutions, but they are outside the scope of this paper.

\subsection{Randomisation sequence}
\label{ssec:randomisation_sequence_Banach}

Recall the random approximating sequence $(U_k)_k$ defined in \eqref{eq:rand_approx_seq}.
In this section, we shall assume that each $\xi_k$ is a $V$-valued stochastic process indexed by $[0,\infty)$, and we shall assume $U_0$ is a $V$-valued random variable.
Below, we shall impose the following regularity assumption on the $(\xi_k)_{k\in\bb{N}_0}$.

For the remainder of Section \ref{sec:Banach}, we shall shorten notation and write $\Vrt{Z}_{\Psi}$ instead of $\Vrt{Z}_{\Psi(\Omega;V)}$ for any $V$-valued random variable $Z$.

\begin{assumption}
    \label{asmp:Noise_regularity_Banach}
    The collection $(\xi_k)_{k\in\bb{N}_0}$ admits an Orlicz norm $\vrt{\cdot}_{\Psi}$ and constants $p\geq 0$ and $0<C_\xi<\infty$, such that for all $k\in\bb{N}_0$ and $t>0$,
    \begin{equation*}
    \Vrt{\xi_k(t)}_{\Psi}\leq C_\xi t^{p+1}.
    \end{equation*}
\end{assumption}

The assumption allows the stochastic processes to be non-Gaussian, to be probabilistically dependent, and to have different distributions and nonzero means.
Furthermore, Assumption \ref{asmp:Noise_regularity_Banach} allows for $\xi_k(t)$ to have different orders of integrability. 
The rates at which the absolute moments decrease to zero as $t$ decreases to zero may differ as well. 
The function $\Psi$ quantifies the maximal common order of integrability, and the parameter $p$ quantifies the maximal common decay rate with respect to $\vrt{\cdot}_{\Psi}$. 

Assumption \ref{asmp:Noise_regularity_Banach} generalises \cite[Assumption 3.3]{LieStuartSullivan2019}, which in turn generalised \cite[Assumption 1]{Conrad2017}. 
 The latter two assumptions considered the $\vrt{\cdot}_{R}$ norm for $R\in\bb{N}$ and the $\vrt{\cdot}_{2}$ norm of $\bb{R}^{d}$-valued random variables respectively. 

We recall the motivation given in \cite{Conrad2017} for the additive random perturbation in \eqref{eq:rand_approx_seq} and in particular for Assumption \ref{asmp:Noise_regularity_Banach}.
Comparing \eqref{eq:det_exact_seq} and \eqref{eq:rand_approx_seq} yields 
\begin{equation*}
    u(t_1) = u(0)+\int_{0}^{h_0} f(s,u(s))\ud s\approx \psi(h_0,0,u(0))+\xi_0(h_0) = U_1.
\end{equation*}
Thus, the random variable $\xi_0(h_0)$ models the uncertainty in the value of the integral term due to the fact that the value of the solution $u$ over the time interval $[0,h_0]$ is known only at time $0$, and not at every time $s$ in the interval $[0,h_0]$.

It is desirable that the approximation above is good with high probability. Given that any reasonable choice of $\psi$ must satisfy $\lim_{h_0\to 0}\psi(h_0,0,u_0)= u_0$, a necessary condition for the approximation above to be good with high probability is that the law of $\xi_0(h_0)$ concentrates around 0 as $h_0\to 0$, because the integral term $\int_{0}^{h_0}f(s,u(s))\ud s\to 0$ as $h_0\to 0$. Using Assumption \ref{asmp:Noise_regularity_Banach} with Markov's inequality yields that for every $\varepsilon>0$,
\begin{equation*}
    \bb{P}(\Abs{\xi_k(t)}_V\geq \varepsilon)\leq \left(\frac{ C_\xi t^{p+1}}{\varepsilon}\right)^r.
\end{equation*}
The inequality above shows that the parameter $p$ quantifies the maximal common rate at which all the laws $(\bb{P}\circ (\Abs{\xi_k(t)}_V)^{-1})_{k}$ contract around the Dirac measure at zero, as $t$ decreases to zero. 

In \cite{Conrad2017,LieStuartSullivan2019}, the parameter $p$ is chosen in order to ensure that the error of the random approximate solution sequence $(U_k)_{k}$ with respect to the exact sequence $(u(t_k))_{k}$ decreases with $h$ at the same rate as the error of the deterministic approximate solution sequence $(u(t_k))_{k}$. This choice is motivated by the goal of showing that probabilistic integrators can have the same convergence rate as the underlying deterministic one-step method. 

 Recall that if $V$ is a separable Banach space and $\mu$ is a Gaussian measure whose support equals $V$, then the Cameron--Martin space of $\mu$ is dense in $V$, and hence there exists a $V$-valued Wiener process $(W(t))_{t\geq 0}$ associated to $\mu$ such that $W(1)\sim \mu$ \cite[Theorem 3.6.1, Proposition 7.2.3]{Bogachev1998Gaussian}.\footnote{The cited results assume centredness of $\mu$, but do not require this property.}
 The next lemma shows that there exists a large class of Gaussian processes that satisfies Assumption \ref{asmp:Noise_regularity_Banach}.
\begin{lemma}
 \label{lem:scaled_Wiener_process_satisfies_noise_regularity_assumption}
 Let $\mu$ be a Gaussian distribution with support equal to $V$, and let $(W(t))_{t\geq 0}$ be a Wiener process associated to $\mu$ such that $W(1)\sim \mu$. 
 Let $\xi$ be a stochastic process on $[0,\infty)$ defined by $t\mapsto \xi(t)\coloneqq t^{p+1/2}W(t)$, and let $(\xi_k)_{k\in\bb{N}_0}$ be i.i.d. copies of $\xi$.
 Then 
 \begin{equation}
 \label{eq:Wiener_processes_satisfy_noise_regularity_asmp_Orlicz}
  \Vrt{\xi(t)}_{\Psi}= \Vrt{\xi(1)}_{\Psi}t^{p+1},
 \end{equation}
 for $\vrt{\cdot}_{\Psi}=\Vrt{\cdot}_{R}$, $R>1$, or $\vrt{\cdot}_{\Psi}=\vrt{\cdot}_{\Psi_2}$, $\Psi_2(z)\coloneqq \exp(z^2)-1$.
 \end{lemma}
\begin{proof}
For $t>0$, we have $\Vrt{\xi(t)}_{\Psi}=t^{p+1/2}\Vrt{W(t)}_{\Psi}=t^{p+1}\Vrt{W(1)}_{\Psi}$.
The first equation follows from the definition of $\xi(t)$, and the second equation follows from the scaling property of the Wiener process, i.e.\ that $W(t)=t^{1/2}W(1)$ in distribution for every $t>0$. 
The conclusion follows since $W(1)=\xi(1)$ as random variables, and because Gaussian random variables are exponentially square integrable by Fernique's theorem.
\end{proof}

\begin{remark}
\label{rem:iid_collection_of_Wiener_processes_satisfies_as3_p_R_regularity}
 The preceding discussion shows that a collection of i.i.d. copies of the standard Wiener process $W$ satisfies Assumption \ref{asmp:Noise_regularity_Banach} with $p=-1/2$, in which case we may set $\xi_k(h_k)$ in \eqref{eq:rand_approx_seq} to be a centred Gaussian random variable with variance proportional to $h_k$. This choice yields a time integration method that resembles methods for stochastic differential equations. However, for the error bound in Theorem \ref{thm:LR_bound_max_error_Banach} below to imply convergence in probability of $(U_n)_n$ to the exact solution sequence $(u(t_k))_{k}$, we need $p>0$. This observation highlights an important difference between the type of time integration methods that we analyse in this paper and time integration methods for stochastic differential equations.
\end{remark}

\subsection{Error bounds}
\label{ssec:error_bounds_Banach}

Recall from \eqref{eq:rand_error_seq_v2} that 
\begin{equation*}
 e_{k+1}=\varphi(h_k,t_k,u(t_k))-\psi(h_k,t_k,U_k)-\xi_k(h_k),\quad k\in[N-1]_0.
\end{equation*}
The following bound is the generalisation of \cite[Theorem 2]{Conrad2017} to our setting.
\begin{lemma}
 \label{lem:max_LR_bound_error_Banach}
 Suppose that 
 \begin{itemize}
  \item Assumption \ref{asmp:Lipschitz_flow_map_Banach} holds with parameters $L_{\varphi}$,
  \item Assumption \ref{asmp:Uniform_truncation_error_Banach} holds with parameters $h^\ast$, $C_{\varphi,\psi}$ and $q$, 
  \item Assumption \ref{asmp:Noise_regularity_Banach} holds with parameters $\Vrt{\cdot}_{\Psi}$, $p$, and $C_\xi$, and
  \item the initial state $U_0$ satisfies $\vrt{U_0}_{\Psi}<\infty$.
 \end{itemize}
 Then for any time grid $(t_k)_k$ such that $0<h\leq h^\ast$, the corresponding error sequence $(e_k)_k$ satisfies
 \begin{equation*}
     \max_k\Vrt{e_k}_{\Psi}\leq \exp(L_\varphi T)\Vrt{e_0}_{\Psi}+\frac{C_{\varphi,\psi}+C_\xi}{L_\varphi}\left(\exp(L_\varphi T)-1\right)h^{p\wedge q}.
 \end{equation*}
 In particular, if $\vrt{e_0}_{\Psi}=0$, then $ \max_k\Vrt{e_k}_{\Psi}=\mathcal{O}(h^{p\wedge q})$.
\end{lemma}
\begin{proof}
It suffices to prove the first statement.
Let $k\in[N-1]_0$.
From \eqref{eq:rand_error_seq_v2} we have
\begin{align}
 \Abs{e_{k+1}}_V
 & \leq \Abs{\varphi(h_k,t_k,u(t_k))-\psi(h_k,t_k,U_k)}_V+\Abs{\xi_k(h_k)}_V
 \nonumber
 \\
 & \leq \Abs{\varphi(h_k,t_k,u(t_k))-\varphi(h_k,t_k,U_k)}_V+\Abs{\varphi(h_k,t_k,U_k)-\psi(h_k,t_k,U_k)}_V
 \nonumber
 \\
 & \phantom{=} \quad +\Abs{\xi_k(h_k)}_V
 \nonumber
 \\
 & \leq (1+L_\varphi h_k)\Abs{e_k}_V+C_{\varphi,\psi}h_k^{q+1}+\Abs{\xi_k(h_k)}_V
 \label{eq:decomp1_bound_e_kplus1_decomposition_using_exact_flow_map}
\end{align}
where \eqref{eq:decomp1_bound_e_kplus1_decomposition_using_exact_flow_map} follows from Assumptions \ref{asmp:Lipschitz_flow_map_Banach} and \ref{asmp:Uniform_truncation_error_Banach}. By taking the $\vrt{\cdot}_{\Psi}$ norm of both sides of \eqref{eq:decomp1_bound_e_kplus1_decomposition_using_exact_flow_map}, using the triangle inequality, Assumption \ref{asmp:Noise_regularity_Banach}, and the bound $h_k\leq h$ from \eqref{eq:time_grid}, we obtain
\begin{equation*}
 \Vrt{e_{k+1}}_{\Psi}\leq (1+L_\varphi h)\Vrt{e_k}_{\Psi}+(C_{\varphi,\psi}+C_\xi)h^{(p\wedge q)+1}.
\end{equation*}
Applying the discrete Gronwall inequality in Lemma \ref{lem:MilsteinTretyakov_GronwallLemma} completes the proof.
\end{proof}

\begin{remark}
 In addition to bounds on the strong error $\vrt{e_k}_{\Psi}$, one can prove bounds on the weak error, i.e.\ bounds of the form
\begin{equation*}
 \vert \mathbb{E}[\Phi(U^h_n)]-\Phi(u_n)\vert \leq Ch^w,
\end{equation*}
for all sufficiently smooth $\bb{R}$-valued functions $\Phi$. Such bounds were proven in \cite[Theorem 2.4]{Conrad2017} and \cite[Section 3]{AbdulleGaregnani2020}, for example. We focus on strong error bounds in this paper.
\end{remark}

To prove Lemma \ref{lem:max_LR_bound_error_Banach}, we take expectations via the $\vrt{\cdot}_{\Psi}$ norm before applying the discrete Gronwall inequality in Lemma \ref{lem:MilsteinTretyakov_GronwallLemma} to conclude. 
By reversing the order of these operations and by using a different discrete Gronwall inequality, we can bound $\vrt{\max_k\abs{e_k}_V}_{\Psi}$.
This yields the result below, which extends \cite[Theorem 3.5]{LieStuartSullivan2019} to our setting.
On one hand, this bound has worse constants than the bound in Lemma \ref{lem:max_LR_bound_error_Banach}.
On the other hand, the bound is stronger, because 
\begin{equation}
\label{eq:stronger_error_bound}
 \max_k\Vrt{e_k}_{\Psi}\leq \vrt{\max_k\Abs{e_k}_V}_{\Psi},
\end{equation}
and because the bound has the same order in $h$ as Lemma \ref{lem:max_LR_bound_error_Banach}.

\begin{theorem}
 \label{thm:LR_bound_max_error_Banach}
 Suppose the hypotheses of Lemma \ref{lem:max_LR_bound_error_Banach} hold.  
 Then for any time grid $(t_k)_k$ with $0<h\leq h^\ast$, the corresponding error sequence $(e_k)_k$ satisfies
 \begin{equation*}
   \Vrt{\max_k \Abs{e_k}_V}_{\Psi}\leq \left(\Vrt{e_0}_{\Psi}+ C_{\varphi,\psi}h^q T+C_\xi h^p T\right)\exp\left(L_\varphi T\right),
 \end{equation*}
 In particular, if $\vrt{e_0}_{\Psi}=0$, then $ \vrt{\max_k\Abs{e_k}_V}_{\Psi}=\mathcal{O}(h^{p\wedge q})$.
\end{theorem}
\begin{remark}
\label{rem:Orlicz_norm_shorter_proof}
When $\Psi(z)=\exp(z^2)-1$, then the strong error bound given in Theorem \ref{thm:LR_bound_max_error_Gelfand} implies the exponential square integrability of the pathwise error $\max_k\Abs{e_k}_V^2$. 
The exponential square integrability of the pathwise error was used in \cite[Section 5]{LieSullivanTeckentrup2018} to establish local Lipschitz continuity of random approximate posteriors --- measured in the Hellinger metric --- with respect to the expected error of the randomised time integrator.
In \cite{LieStuartSullivan2019}, exponential integrability was obtained by considering $\vrt{\max_k\Abs{e_k}_V}_{R}$ for all $R\in\bb{N}$ and using the series representation of the exponential function. 
The use of Orlicz norms allows us to exploit the fact that the random approximating sequence $(U_k)_k$ inherits the integrability properties of the collection $(\xi_k)_k$. 
This leads to a simpler proof of exponential integrability.
\end{remark}

\begin{proof}[Proof of Theorem \ref{thm:LR_bound_max_error_Banach}]
Using \eqref{eq:decomp1_bound_e_kplus1_decomposition_using_exact_flow_map} and applying the discrete Gronwall inequality in Lemma \ref{lem:Special_Gronwall_Inequality_corollary}, we obtain for every $k\in[N-1]_0$ that
\begin{equation*}
\Abs{e_{k+1}}_V\leq \left(\Abs{e_0}_V+\sum_{k \in[N-1]_0} \left(C_{\varphi,\psi}h_k^{q+1}+\Abs{\xi_k(h_k)}_V\right)\right)\exp\left(\sum_{0\leq j\leq k} L_\varphi h_j\right).
\end{equation*}
Since the sum in the exponential increases with $k$, setting $k=N-1$ above and using \eqref{eq:time_grid} to obtain $\sum_{j\in[N-1]_0}h_j=T$ yields the `pathwise' bound
\begin{equation}
 \label{eq:pathwise_bound_max_k_error_Banach}
 \max_k\Abs{e_k}_V\leq \left(\Abs{e_0}_V+ C_{\varphi,\psi}h^{q}T+\sum_{k\in[N-1]_0}\Abs{\xi_k(h_k)}_V\right)\exp\left(L_\varphi T\right).
\end{equation}
By taking the $\vrt{\cdot}_{\Psi}$ norm of both sides of \eqref{eq:pathwise_bound_max_k_error_Banach}, the triangle inequality, Assumption \ref{asmp:Noise_regularity_Banach}, and \eqref{eq:upper_bound_on_sum_time_steps_to_power_plus_one}, we obtain
\begin{align*}
 \Vrt{\max_k \Abs{e_k}_V}_{\Psi}\leq &\left(\Vrt{e_0}_{\Psi}+C_{\varphi,\psi}h^{q}T+\sum_{k\in[N-1]_0}\Vrt{\xi_k(h_k)}_{\Psi}\right)\exp\left(L_\varphi T\right)
  \\
 \leq &\left(\Vrt{e_0}_{\Psi}+C_{\varphi,\psi}h^qT+C_\xi h^p T\right)\exp\left(L_\varphi T\right),
\end{align*}
which completes the proof.
\end{proof}

\begin{remark}
\label{rem:expected_order}
Under the assumption that $V=\bb{R}^{d}$ and under the assumption that the randomisation sequence $(\xi_k(h_k))_k$ consists of centred, independent random variables, \cite[Theorem 2]{Conrad2017} and \cite[Theorem 3.4]{LieStuartSullivan2019} consider the special case where $\vrt{\cdot}_{\Psi}=\vrt{\cdot}_{2}$ in Lemma \ref{lem:max_LR_bound_error_Banach} and Theorem \ref{thm:LR_bound_max_error_Banach}, and establish $\mathcal{O}(h^{q\wedge (p+1/2)})$ bounds on the strong error respectively. 
The order in these bounds is better than the bounds we proved above. 
However, both the proofs of these results exploit both the inner product structure of $\bb{R}^{d}$ and the fact that linear functionals of the $\xi_k$ appear in the expansion of $\abs{e_{k+1}}^2_{\bb{R}^d}$. In the key inequality \eqref{eq:decomp1_bound_e_kplus1_decomposition_using_exact_flow_map}, we cannot exploit an inner product even if it were available, because we only consider $\abs{e_{k+1}}_{V}$.
In Section \ref{ssec:higher_order_error_bounds_independent_centred_Gelfand}, we shall generalise \cite[Theorem 2]{Conrad2017} and \cite[Theorem 3.4]{LieStuartSullivan2019} from $\bb{R}^{d}$ to general Hilbert spaces.
\end{remark}


\section{Variational setting}
\label{sec:Gelfand}

For evolution equations originating from PDEs with possibly non-smooth right hand sides or non-smooth initial conditions, the classical solution theory that we considered in Section \ref{sec:Banach} might not apply, because the requirement that the operator $f$ in \eqref{eq:operator_differential_equation} satisfies $f(t,v)\in  V$ for every $v\in V$ and all suitable $t$ might be too strong. 
For example, this requirement does not hold for the heat equation in Sobolev spaces $W^{k,p}$. 
There are several settings that extend the classical setting for such problems. 
In this section, we focus on the variational setting, because it is suitable for numerical time integration methods. 
In the variational setting, we consider a Gelfand triplet $V \hookrightarrow H \simeq H' \hookrightarrow V'$, which is a sequence of continuous embeddings of a Banach space $V$ into a Hilbert space $H$ that is identified with its dual space $H'$, which is then embedded in the dual space $V'$ of $V$ \cite[Proposition 23.13]{ZeidlerIIA}. 

In this section, we further specify the operator differential equation \eqref{eq:operator_differential_equation} to be
\begin{equation}\label{eq:pde}
    u(0)=\vartheta\in H,\quad u'(t) + A(t,u(t)) = b(t)\in V',\quad t\in [0,T]
\end{equation}
for a given operator $A \colon [0,T] \times V\rightarrow V'$ and $b \in L^{p'}(0,T;V') $. 
The equation \eqref{eq:pde} is written in the form that is common in PDE theory instead of the form used in \eqref{eq:operator_differential_equation}, where the right-hand side would be defined by $f(t,u(t)) \coloneqq b(t) -A(t,u(t))$. 
The solution of \eqref{eq:pde} belongs to the space
\begin{equation*}
    \mathcal{W}^p (0,T) \coloneqq \left\{ u \in L^p (0,T;V) \, \middle| \, u' \in L^{p'}(0,T;V') \, \text{with } \frac{1}{p} + \frac{1}{p'} = 1\right\} ,
\end{equation*}
which is continuously embedded into $C([0,T];H)$ \cite[Satz 8.4.1]{Emmrich2004}. 
We emphasise that a solution of \eqref{eq:pde} must satisfy the equation only for almost every $t\in [0,T]$, and not for every $t$.

There are several conditions --- e.g.\ Lipschitz or one-sided  Lipschitz conditions, strong positivity, monotonicity, or coercivity --- that one can impose on $A$ and $b$ in order to guarantee the existence of a unique variational solution $u \in \mathcal{W}^p(0,T) \hookrightarrow C([0,T];H)$ \cite[Prop. 23.23]{ZeidlerIIA}.
Under stronger assumptions, higher regularity of $u$ can be achieved \cite[Satz 8.5.1]{Emmrich2004}. In some cases, the flow map is continuous and even Lipschitz; see \cite[Theorem 23.A]{ZeidlerIIA} for linear problems and \cite[Corollary 23.26]{ZeidlerIIA} for the time-dependent case.

Recall the definition \eqref{eq:det_exact_seq} of the sequence $(u(t_k))_{k\in[N]_0}$ of states of the exact solution:
\begin{equation*}
 u(t_{k+1})=\varphi(h_k,t_k,u(t_k)),\quad k\in[N-1]_0,
\end{equation*}
where $\varphi$ is the flow map associated to the differential equation of interest \eqref{eq:pde}.
In the variational setting, the flow map $\varphi$ is a mapping $\varphi \colon [0,h^\ast]\times[0,T] \times H \rightarrow H$. 
Next, recall that $\psi$ is the approximate flow map associated to a time integration method, and that according to \eqref{eq:rand_approx_seq}, we construct the random approximating sequence $(U_k)_{k\in[N]_0}$ according to
\begin{equation*}
 U_{k+1}=\psi(h_k,t_k,U_k)+\xi_k(h_k),\quad k\in[N-1]_0.
\end{equation*}
In this section, we shall assume that the initial condition $U_0$ is a $H$-valued random variable, and that each $\xi_k$ is a $H$-valued stochastic process indexed by $[0,\infty)$.

We shall make the following assumptions on $\psi$.
\begin{assumption}
    \label{asmp:local_truncation_error_Lipschitz_numerical_flow_Gelfand}
    Let $h^\ast>0$, and let $\psi \colon [0,h^\ast]\times [0,T]\times H\to V$ satisfy the following conditions:
    \begin{enumerate}
        \item There exists a scalar $q\geq 0$, a function $C_{\varphi,\psi} \colon [0,T]\times H\to (0,\infty)$ that is bounded on bounded subsets, and a dense subset $\mathcal{D}\subset H$, such that, for every $h\in [0,h^\ast]$ and for every $(t,x)\in [0,T-h]\times H$ with $x=\varphi(s,0,\vartheta')$ for some $s\geq 0$ and $\vartheta'\in\mathcal{D}$,
        \begin{equation}
            \label{eq:local_truncation_error_gelfand}
            \Abs{\varphi(h,t,x)-\psi(h,t,x)}_H\leq C_{\varphi,\psi}(t,x)h^{q+1};
        \end{equation}
        \item There exists a constant $L_{\psi}>0$ such that for all $(h,t)\in [0,h^\ast]\times [0,T]$ and for any $x,y\in H$,
        \begin{equation}
            \label{eq:global_Lipschitz_approx_flow_gelfand}
            \Abs{\psi(h,t,x)-\psi(h,t,y)}_H\leq(1+ L_{\psi}h)\Abs{x-y}_H.
        \end{equation}
    \end{enumerate}
\end{assumption}
The first statement of Assumption \ref{asmp:local_truncation_error_Lipschitz_numerical_flow_Gelfand} means that the one-step error bound \eqref{eq:local_truncation_error_gelfand} holds for any $x$ that lies on some solution $u\in C([0,T];H)$ of \eqref{eq:pde}, where the initial condition $\vartheta'=u(0)$ belongs to the dense subset $\mathcal{D}$.
We make the hypothesis of density in order to account for known results concerning error bounds for time integration of PDEs, see e.g. \cite[Chapter 7]{Thomee2006}.

The local truncation error \eqref{eq:local_truncation_error_gelfand} is a reasonable requirement for any deterministic time integration method $\psi$ and weakens the uniform local truncation error bound of Assumption \ref{asmp:Uniform_truncation_error_Banach}.
Given \eqref{eq:local_truncation_error_gelfand}, we define 
  \begin{equation}
   \label{eq:C_Phi_Psi_infty}
    \Vrt{C_{\varphi,\psi}}_{\infty}\coloneqq \sup_{t\in[0,T]}C_{\varphi,\psi}(t,u(t)),
  \end{equation}
  for any solution $u$ of \eqref{eq:pde} with initial condition $\vartheta\in\mathcal{D}$.
Since the solution $u$ of \eqref{eq:pde} belongs to $ C([0,T];H)$, it is a bounded set.
Hence, the first statement of Assumption \ref{asmp:local_truncation_error_Lipschitz_numerical_flow_Gelfand} ensures the finiteness of $\vrt{C_{\varphi,\psi}}_{\infty}$.
The second statement of Assumption \ref{asmp:local_truncation_error_Lipschitz_numerical_flow_Gelfand} describes a global Lipschitz continuity property of the approximate flow map $\psi$ with respect to the third argument of the map $\psi$.
  
  For the error bounds that we prove in this section, the bounds \eqref{eq:local_truncation_error_gelfand} and \eqref{eq:global_Lipschitz_approx_flow_gelfand} shall play the roles of Assumptions \ref{asmp:Uniform_truncation_error_Banach} and \ref{asmp:Lipschitz_flow_map_Banach} respectively in the error bounds of Section \ref{ssec:error_bounds_Banach}.

Next, we formulate the analogue of Assumption \ref{asmp:Noise_regularity_Banach} for the collection $(\xi_k)_{k\in\bb{N}_0}$ of stochastic processes.
For the remainder of Section \ref{sec:Gelfand}, we shall simplify notation and write $\Vrt{Z}_{\Psi}$ instead of $\Vrt{Z}_{\Psi(\Omega;H)}$ for any $H$-valued random variable $Z$.
\begin{assumption}
    \label{asmp:Noise_regularity_Gelfand}
    The collection $(\xi_k)_{k\in\bb{N}_0}$ admits an Orlicz norm $\vrt{\cdot}_{\Psi}$ and constants $p\geq 0$ and $0<C_\xi<\infty$, such that for all $k\in\bb{N}_0$ and $t>0$,
    \begin{equation*}
    \Vrt{\xi_k(t)}_{\Psi}\leq C_\xi t^{p+1}.
    \end{equation*}
\end{assumption}
The only difference between Assumption \ref{asmp:Noise_regularity_Gelfand} and Assumption \ref{asmp:Noise_regularity_Banach} is that the stochastic processes are $H$-valued instead of $V$-valued.

\subsection{$L^2$-error bounds for independent and centred randomisation}
\label{ssec:L2_error_bounds_independent_centred_case_Gelfand}

In this section, we assume that the $(\xi_k)_{k}$ are mutually independent and centred stochastic processes.
In particular, for any time grid \eqref{eq:time_grid}, the corresponding random variables $(\xi_k(h_k))_{k\in[N-1]_0}$ are mutually independent and centred.
We shall generalise the $L^2$-error bounds from \cite[Theorem 2]{Conrad2017} and \cite[Theorem 3.4]{LieStuartSullivan2019} to the variational setting.

For any time grid $(t_k)_{k\in[N]_0}$ and $k\in[N-1]_0$, let $\mathcal{F}_k\coloneqq \sigma(\xi_j(h_j): j\in[ k]_0)$, i.e. $(\mathcal{F}_k)_{k\in[N-1]_0}$ is the filtration generated by the randomisation sequence $(\xi_j(h_j))_{j\in[N-1]_0}$.

The following lemma only requires mutual independence of the $(\xi_\ell)_{\ell}$.
\begin{lemma}
\label{lem:Uk_measurable_wrt_Fk}
Suppose that Assumption \ref{asmp:local_truncation_error_Lipschitz_numerical_flow_Gelfand} holds.
Let $(t_k)_{k\in[N]_0}$ be an arbitrary time grid.
Then for $j\in [N-1]_0$, $U_{j+1}$ is a measurable function of $U_0$ and $\{\xi_\ell(h_\ell)\ :\ \ell\in [j]_0\}$. In particular, if the $(\xi_\ell)_{\ell}$ are mutually independent, then for every $j\in[N-1]$, $\xi_j(h_j)$ and $U_j$ are independent, and $\xi_j(h_j)$ is independent of $\mathcal{F}_j$.
\end{lemma}
\begin{proof}
 It follows from \eqref{eq:global_Lipschitz_approx_flow_gelfand} in Assumption \ref{asmp:local_truncation_error_Lipschitz_numerical_flow_Gelfand} that, for arbitrary $(h,t)$, $\psi(h,t,z)$ is globally Lipschitz continuous with respect to $z \in H$. 
 Hence, $U_{j+1}$ is a measurable function of $U_{j}$ and $\xi_j(h_j)$, for every $j\in [K-1]_0$. 
 This proves the first statement. The second statement follows from the first and the definition of $\mathcal{F}_j$.
\end{proof}

The following result is the generalisation of \cite[Theorem 2.2]{Conrad2017}, which considered the case $H=\bb{R}^{d}$ for $d\in\bb{N}$.
\begin{lemma}
\label{lem:max_L2_error_independent_centred_noise_Gelfand}
Suppose the following statements are true:
\begin{itemize}
 \item Assumption \ref{asmp:local_truncation_error_Lipschitz_numerical_flow_Gelfand} holds with parameters $h^\ast$, $q$, $C_{\varphi,\psi}$, $\mathcal{D}$, and $L_\psi$,
 \item Assumption \ref{asmp:Noise_regularity_Gelfand} holds with parameters $\Vrt{\cdot}_{\Psi}\coloneqq \Vrt{\cdot}_{2}$, $p$, and $C_\xi$,
 \item the $(\xi_j)_{j}$ are mutually independent and centred, and
 \item the initial condition $\vartheta$ of \eqref{eq:pde} belongs to $\mathcal{D}$, and $\vrt{U_0}_{2}<\infty$.
\end{itemize}
Then there exists a $L'_{\psi}>0$ depending only on $L_\psi$, such that for any time grid $(t_k)_k$ satisfying $0<h\leq 1\wedge h^\ast$, the associated error sequence $(e_k)_k$ satisfies
\begin{equation*}
 \max_{k}\Vrt{e_k}^2_{2}\leq \left(\Vrt{e_0}^2_{2}+  3T \Vrt{C_{\varphi,\psi}}_\infty^2 T h^{2q}+C_\xi^2 Th^{2p+1}\right)\exp\left(L'_\psi T\right).
\end{equation*}
In particular, if $\vrt{e_0}_{2}=0$, then $\max_{k\in[N]_0}\Vrt{e_k}_{2}=\mathcal{O}(h^{q\wedge (p+1/2)})$.
\end{lemma}
We state the proof below, even though it is very similar to the proof of \cite[Theorem 2.2]{Conrad2017}.
This is because the proof will be useful later in Section \ref{ssec:higher_order_error_bounds_independent_centred_Gelfand}, where we discuss the feasibility of bounding $\max_{k} \vrt{e_k}_{R}$ for $R> 2$ under similar assumptions as Lemma \ref{lem:max_L2_error_independent_centred_noise_Gelfand}.
An important difference between our proof and the proof of \cite[Theorem 2.2]{Conrad2017} is that the latter assumes uniform truncation error, e.g.\ as in Assumption \ref{asmp:Uniform_truncation_error_Banach}. 
Instead, we use Assumption \ref{asmp:local_truncation_error_Lipschitz_numerical_flow_Gelfand}.

\begin{proof}[Proof of Lemma \ref{lem:max_L2_error_independent_centred_noise_Gelfand}]
Let $k\in[N-1]_0$. By the definition \eqref{eq:rand_error_seq_v2} of the error sequence $(e_k)_{k\in[N]_0}$, 
 \begin{align}
  \Abs{e_{k+1}}_H^2
  & = \Abs{\varphi(h_k,t_k,u(t_k))-\psi(h_k,t_k,U_k)}_H^2+\Abs{\xi_k(h_k)}_H^2
  \label{eq:bound_abs_ekplus1_squared_preliminary}
  \\
  & \phantom{=} \quad + 2\Ang{\varphi(h_k,t_k,u(t_k))-\psi(h_k,t_k,U_k), \xi_k(h_k)}_H.
 \nonumber
 \end{align}
Recall the term $\Vrt{C_{\varphi,\psi}}_{\infty}$ from \eqref{eq:C_Phi_Psi_infty}. We obtain
 \begin{align}
  & \Abs{\varphi(h_k,t_k,u(t_k))-\psi(h_k,t_k,U_k)}_H^2
  \nonumber
  \\
  & \quad = \Abs{\varphi(h_k,t_k,u(t_k))-\psi(h_k,t_k,u(t_k))-\psi(h_k,t_k,u(t_k))-\psi(h_k,t_k,U_k)}^2_H 
  \nonumber
  \\
  & \quad \leq \left(1+\left(\tfrac{2}{h_k}\right)\right) C_{\varphi,\psi}(t_k,u(t_k))^2 h_{k}^{2q+2}+(1+2h_k)(1+L_\psi h_k)^2\Abs{e_k}_H^2
  \nonumber
  \\
  & \quad \leq 3 \Vrt{C_{\varphi,\psi}}_{\infty}^2 h_{k}^{2q+1}+(1+2h_k)(1+L_\psi h_k)^2\Abs{e_k}_H^2.
  \nonumber
  \end{align}
  The first inequality follows from the hypothesis that the initial condition $\vartheta$ of \eqref{eq:pde} belongs to $\mathcal{D}$, since we can then apply the local truncation error bound \eqref{eq:local_truncation_error_gelfand} of Assumption \ref{asmp:local_truncation_error_Lipschitz_numerical_flow_Gelfand} and Young's inequality. The second inequality follows from the fact that $h_k\leq h\leq 1$.
  
  Let $(a_i)_{i=0}^{3}$ be the coefficients of the polynomial $h'\mapsto (1+2h')(1+L_\psi h')^2$ and let $L'_\psi\coloneqq \sum_{i=1}^{3}a_i$.
  Then $L'_\psi$ depends only on $L_\psi$ and not on the time grid.
  By the preceding math display and the hypothesis that $h_k\leq h\leq 1$,
    \begin{equation}
       \label{eq:squared_bound_abs_Phi_utk_minus_Psi_Uk}
     \Abs{\varphi(h_k,t_k,u(t_k))-\psi(h_k,t_k,U_k)}_H^2\leq 3 \Vrt{C_{\varphi,\psi}}_{\infty}^2 h_{k}^{2q+1}+(1+L'_\psi h_k)\Abs{e_k}_H^2.
    \end{equation}
 Substituting \eqref{eq:squared_bound_abs_Phi_utk_minus_Psi_Uk} into the bound  \eqref{eq:bound_abs_ekplus1_squared_preliminary} on $\abs{e_{k+1}}_H^2$ yields 
 \begin{align}
  \Abs{e_{k+1}}_H^2\leq &\left(3 \Vrt{C_{\varphi,\psi}}_{\infty}^2 h_{k}^{2q+1}+(1+L'_\psi h_k)\Abs{e_k}_H^2\right)+\Abs{\xi_k(h_k)}_H^2
  \label{eq:bound_abs_ekplus1_squared}
  \\
  &+ 2\Ang{\varphi(h_k,t_k,u(t_k))-\psi(h_k,t_k,U_k), \xi_k(h_k)}_H.
  \nonumber
 \end{align}
 By mutual independence of the $(\xi_j(h_j))_{j \in [N-1]_0}$, it follows from the second statement of Lemma \ref{lem:Uk_measurable_wrt_Fk} that the arguments of the inner product are independent. 
 By taking expectations of \eqref{eq:bound_abs_ekplus1_squared} and centredness of the $(\xi_j(h_j))_{j \in [N-1]_0}$, the expectation of the inner product vanishes. 
 By Assumption \ref{asmp:Noise_regularity_Gelfand}, we have
 \begin{equation*}
  \Vrt{e_{k+1}}^2_2\leq (1+L'_\psi h_k)\Vrt{e_k}^2_2+3\Vrt{C_{\varphi,\psi}}_\infty^2h_k^{2q+1}+ C_\xi^2 h_k^{2p+2}.
  \end{equation*}
  Using the discrete Gronwall inequality in Lemma \ref{lem:Special_Gronwall_Inequality_corollary} and \eqref{eq:upper_bound_on_sum_time_steps_to_power_plus_one} completes the proof.
\end{proof}

  We shall use the next result, Lemma \ref{lem:martingale}, to prove Proposition \ref{prop:L2_max_error_bound_independent_centred_noise_Gelfand} below.
  A similar result to Lemma \ref{lem:martingale} was established in the proof of \cite[Theorem 3.4]{LieStuartSullivan2019}, under the assumption that $\psi$ preserves square integrability of random variables, i.e.\ that $\psi(Z)\in L^2(\Omega;\bb{R}^d)$ for every $Z\in L^2(\Omega;\bb{R}^d)$. 
  Lemma \ref{lem:martingale} removes this assumption, by using Lemma \ref{lem:max_L2_error_independent_centred_noise_Gelfand}.
 \begin{lemma}
  \label{lem:martingale}
  Suppose the hypotheses of Lemma \ref{lem:max_L2_error_independent_centred_noise_Gelfand} hold. Then for any time grid $(t_j)_{j\in[N]_0}$ with $h>0$, the stochastic process $(M_k)_{k\in[N-1]_0}$ defined by
  \begin{equation}
   \label{eq:Mk_martingale}
   M_k\coloneqq \sum_{j=0}^{k}\Ang{\varphi(h_j,t_j,u(t_j))-\psi(h_j,t_j,U_j),\xi_j(h_j)}_H
  \end{equation}
  is a $\bb{R}$-valued, square-integrable martingale with respect to $(\mathcal{F}_k)_{k\in[N-1]_0}$. 
  If in addition the time grid $(t_j)_{j\in[N]_0}$ satisfies $h\leq 1\wedge h^\ast$, then there exists a universal constant $\kappa>0$ such that for every $k\in[N-1]_0$,
  \begin{equation}
   \label{eq:consequence_BDG_inequality_Mk}
   \bb{E}\left[\max_{j\in[k]_0}\Abs{M_k} \right]\leq   \Vrt{C_{\varphi,\psi}}_{\infty}^2h^{2q+1}+\frac{1}{4}\bb{E}\left[\max_{j\in[k]_0}\Abs{e_j}^2_H\right]+\kappa^2(1+L'_\psi) TC_\xi^2 h^{2p+1},
  \end{equation}
  for the same $L'_\psi$ given in Lemma \ref{lem:max_L2_error_independent_centred_noise_Gelfand}.
 \end{lemma}
 \begin{proof}
  See Section \ref{ssec:proof_of_lemma_martingale} for the proof.
 \end{proof}

 Next, we use Lemma \ref{lem:martingale} to prove the following error bound, which is stronger than the bound given in Lemma \ref{lem:max_L2_error_independent_centred_noise_Gelfand} because of \eqref{eq:stronger_error_bound}.
\begin{proposition}
\label{prop:L2_max_error_bound_independent_centred_noise_Gelfand}
Suppose the hypotheses of Lemma \ref{lem:max_L2_error_independent_centred_noise_Gelfand} hold. Then for any time grid $(t_k)_{k}$ with $0<h\leq  1\wedge h^\ast$, the corresponding error sequence $(e_k)_k$ satisfies
\begin{align*}
 &\Vrt{\max_{k}\Abs{e_k}_H}_{2}^2
 \\
  & \quad \leq 2\left(\Vrt{e_0}^2_{2}+4\Vrt{C_{\varphi,\psi}}_{\infty}^2 h^{2q}T+ C_\xi^2 Th^{2p+1} (1+\kappa^2(1+L'_\psi))\right)\exp\left(2L'_\psi T\right),
\end{align*}
for the universal constant $\kappa$ in \eqref{eq:consequence_BDG_inequality_Mk} and the constant $L'_\psi$ given in Lemma \ref{lem:max_L2_error_independent_centred_noise_Gelfand}.
In particular, if $\vrt{e_0}_{2}=0$, then $\Vrt{\max_{k}\Abs{e_k}_H}_{2}=\mathcal{O}(h^{q\wedge (p+1/2)})$.
 \end{proposition}
 \begin{proof}
   See Section \ref{ssec:proof_of_proposition_L2_max_error} for the proof.
 \end{proof}

\subsection{Error bounds of higher integrability order for independent and centred randomisation}
\label{ssec:higher_order_error_bounds_independent_centred_Gelfand}

It is natural to ask if one can prove the analogues of Lemma \ref{lem:max_L2_error_independent_centred_noise_Gelfand} or Proposition \ref{prop:L2_max_error_bound_independent_centred_noise_Gelfand} where we use $\vrt{\cdot}_{R}$, $R>2$, while keeping the same order in $h$.
Suppose that we wish to prove the analogue of Lemma \ref{lem:max_L2_error_independent_centred_noise_Gelfand} for $R=3$.
It follows from the triangle inequality and the definition \eqref{eq:rand_error_seq_v2} that 
\begin{equation*}
 \Abs{e_{k+1}}_H\leq \Abs{\varphi(h_k,t_k,u(t_k))-\psi(h_k,t_k,U_k)}_H+\Abs{\xi_k(h_k)}_H. 
\end{equation*}
Thus
\begin{equation*}
 \Abs{e_{k+1}}_H^3\leq \Abs{e_{k+1}}_H^2\left(\Abs{\varphi(h_k,t_k,u(t_k))-\psi(h_k,t_k,U_k)}_H+\Abs{\xi_k(h_k)}_H\right)
\end{equation*}
and substituting \eqref{eq:bound_abs_ekplus1_squared_preliminary} results in an upper bound on $\Abs{e_{k+1}}_H^3$ containing the mixed product of an inner product term and a norm term,
\begin{equation*}
 \Ang{\varphi(h_k,t_k,u(t_k))-\psi(h_k,t_k,U_k),\xi_k(h_k)}_H\Abs{\varphi(h_k,t_k,u(t_k))-\psi(h_k,t_k,U_k)}_H.
\end{equation*}
In general, this product will not vanish in expectation, because one can no longer exploit the commutativity of the inner product with the expectation operator. 
The same assertion is valid for $R\geq 3$.
This is the important difference between the $R=2$ case that was proven in Lemma \ref{lem:max_L2_error_independent_centred_noise_Gelfand} and the case $R\geq 3$. This difference implies that we must use the Cauchy--Schwarz inequality to bound products.
Using the Cauchy--Schwarz inequality yields
\begin{equation*}
 \Abs{e_{k+1}}_H^3\leq \sum_{i=0}^{3}\begin{pmatrix} 3 \\ i\end{pmatrix} \Abs{\varphi(h_k,t_k,u(t_k))-\psi(h_k,t_k,U_k)}_H^i\Abs{\xi_k(h_k)}_H^{3-i}.
\end{equation*}
We can obtain the same bound by applying the binomial theorem to the bound $\Abs{e_{k+1}}_H\leq \Abs{\varphi(h_k,t_k,u(t_k))-\psi(h_k,t_k,U_k)}_H+\Abs{\xi_k(h_k)}_H$.

If the stochastic processes $(\xi_k)_k$ are mutually independent, then we may use the second statement of Lemma \ref{lem:Uk_measurable_wrt_Fk}.
Assuming that $e_0=0$ almost surely and taking expectations of the summand for $i=2$ yields
\begin{align*}
& \bb{E}\left[\Abs{\varphi(h_k,t_k,u(t_k))-\psi(h_k,t_k,U_k)}_H^2\Abs{\xi_k(h_k)}_H\right]
\\
 &  = \bb{E}\left[\Abs{\varphi(h_k,t_k,u(t_k))-\psi(h_k,t_k,U_k)}_H^2\right]\bb{E}\left[\Abs{\xi_k(h_k)}_H\right] & \text{by independence}
 \\
 & \leq \left( \mathcal{O}(h_{k}^{2q+1})+(1+L'_\psi h_k)\bb{E}\left[\Abs{e_k}_H^2\right]\right) C_\xi h_k^{p+1} &\text{by \eqref{eq:squared_bound_abs_Phi_utk_minus_Psi_Uk}, Assumption \ref{asmp:Noise_regularity_Gelfand}}
 \\
 & \leq \left( \mathcal{O}(h_{k}^{2q+1})+ \mathcal{O}(h^{2q})+\mathcal{O}(h^{2p+1})\right)C_\xi h_k^{p+1}&\text{by Lemma \ref{lem:max_L2_error_independent_centred_noise_Gelfand}.}
\end{align*}
This yields a bound on $\Vrt{e_{k+1}}_3^{3}$ by a term that is $\mathcal{O}(h^{(2q)\wedge(2p+1)+p+1})$. 
Applying a discrete Gronwall inequality produces a bound on $\max_{k\in[N]}\vrt{e_k}_{3}^{3}$ that is $\mathcal{O}(h^{(2q)\wedge(2p+1)+p})$.
Since this upper bound on the exponent arises from the mixed product mentioned above, and since such mixed products will arise in any expansion of $\abs{e_{k+1}}_H^R$, we cannot expect to prove that $\max_{k}\vrt{e_k}_{R}=\mathcal{O}(h^{q\wedge (p+1/2)})$ for $R>2$ using the techniques that we applied earlier, even if the $(\xi_k)_{k}$ are mutually independent and centred.

For the $L^3$ analogue of Proposition \ref{prop:L2_max_error_bound_independent_centred_noise_Gelfand}, the fact that terms involving inner products do not vanish in expectation also poses a problem. 
This is because the proof of the $L^2$ case in Proposition \ref{prop:L2_max_error_bound_independent_centred_noise_Gelfand} relies on the bound \eqref{eq:consequence_BDG_inequality_Mk} in Lemma \ref{lem:martingale} on the martingale $(M_k)_k$. 
This bound in turn follows from the Burkholder--Davis--Gundy inequality for martingales \cite[Chapter IV, \S 4, Theorem (4.1)]{RevuzYor2009}. 
For the $L^3$ case, the expectations of products containing an inner product term do not vanish, because one can no longer exploit commutativity of the inner product with the expectation operator, due to the mixed product.
As a result, the martingale $(M_k)_k$ does not appear, and one cannot apply the Burkholder--Davis--Gundy inequality to prove a bound similar to \eqref{eq:consequence_BDG_inequality_Mk}. 
Instead, one must apply the Cauchy--Schwarz inequality or the binomial theorem, as we did above. 
This results in a bound on $\vrt{\max_k\Abs{e_k}_H}_{3}$ that is worse than $\mathcal{O}(h^{q\wedge (p+1/2)})$.

\subsection{Error bounds of higher integrability order without independence or centredness assumptions}
\label{ssec:higher_order_error_bounds_Gelfand}

In this section, we prove a strong error bound for a general Orlicz norm instead of for the $\vrt{\cdot}_2$-norm.
We use the same hypotheses as for Lemma \ref{lem:max_L2_error_independent_centred_noise_Gelfand} and Proposition \ref{prop:L2_max_error_bound_independent_centred_noise_Gelfand}, except that we do not assume mutual independence or centredness of the stochastic processes $(\xi_k)_{k\in\bb{N}_0}$.
\begin{theorem}
 \label{thm:LR_bound_max_error_Gelfand}
Suppose the following statements are true:
\begin{itemize}
 \item Assumption \ref{asmp:local_truncation_error_Lipschitz_numerical_flow_Gelfand} holds with parameters $h^\ast$, $q$, $C_{\varphi,\psi}$, $\mathcal{D}$, and $L_\psi$,
 \item Assumption \ref{asmp:Noise_regularity_Gelfand} holds with parameters $\Vrt{\cdot}_{\Psi}$, $p$, and $C_\xi$, and
 \item the initial condition $\vartheta$ of \eqref{eq:pde} belongs to $\mathcal{D}$, and $\vrt{U_0}_{\Psi}<\infty$.
\end{itemize}
Then for any time grid $(t_k)_k$ with $0<h\leq h^\ast$, the corresponding error sequence $(e_k)_k$ satisfies
 \begin{equation*}
  \Vrt{\max_k\Abs{e_k}_H}_{\Psi} \leq   \left(\Vrt{e_0}_{\Psi}+\Vrt{C_{\varphi,\psi}}_{\infty}h^q T+C_\xi h^p T\right)\exp\left(L_\psi T\right).
 \end{equation*}
\end{theorem}
In the results from Section \ref{ssec:higher_order_error_bounds_independent_centred_Gelfand}, we required that the maximal time step $h$ associated to the time grid satisfies $h\leq 1\wedge h^\ast$.
In Theorem \ref{thm:LR_bound_max_error_Gelfand}, we only require that $h\leq h^\ast$.
The discussion of exponential integrability in Remark \ref{rem:Orlicz_norm_shorter_proof} also applies to Theorem \ref{thm:LR_bound_max_error_Gelfand}.
\begin{proof}[Proof of Theorem \ref{thm:LR_bound_max_error_Gelfand}]
 Recall \eqref{eq:rand_error_seq_v2}:
\begin{equation*}
 e_{k+1}=\varphi(h_k,t_k,u(t_k))-\psi(h_k,t_k,U_k)-\xi_k(h_k),\quad k\in[N-1]_0.
\end{equation*}
By the triangle inequality, and by \eqref{eq:local_truncation_error_gelfand} and \eqref{eq:global_Lipschitz_approx_flow_gelfand} from Assumption \ref{asmp:local_truncation_error_Lipschitz_numerical_flow_Gelfand},
\begin{align*}
 & \Abs{\varphi(h_k,t_k,u(t_k))-\psi(h_k,t_k,U_k)}_H
 \\
 & \quad \leq \Abs{\varphi(h_k,t_k,u(t_k))-\psi(h_k,t_k,u(t_k))}_H+\Abs{\psi(h_k,t_k,u(t_k))-\psi(h_k,t_k,U_k)}_H
 \\
 & \quad \leq \Vrt{C_{\varphi,\psi}}_{\infty}h_k^{q+1}+(1+L_\psi h_k)\abs{e_k}_H.
\end{align*}
From this it follows that 
\begin{equation}
\label{eq:decomp1_bound_e_kplus1_decomposition_using_numerical_flow_map}
 \Abs{e_{k+1}}_H\leq \Vrt{C_{\varphi,\psi}}_{\infty}h_k^{q+1}+(1+L_\psi h_k)\Abs{e_k}_H+\Abs{\xi_k(h_k)}_H.
\end{equation}
Applying Lemma \ref{lem:Special_Gronwall_Inequality_corollary} and using the same arguments that yielded \eqref{eq:pathwise_bound_max_k_error_Banach}, we obtain the analogous pathwise bound
\begin{align*}
 \max_k\Abs{e_k}_H\leq \left(\Abs{e_0}_H+\Vrt{C_{\varphi,\psi}}_{\infty}h^{q}T+\sum_{k\in[N-1]_0}\Abs{\xi_k(h_k)}_H\right)\exp\left(L_\psi T\right).
\end{align*}
Taking the $\vrt{\cdot}_{\Psi}$ norm of both sides and applying Assumption \ref{asmp:Noise_regularity_Gelfand} completes the proof.
\end{proof}
\begin{remark}
 The inequality \eqref{eq:decomp1_bound_e_kplus1_decomposition_using_numerical_flow_map} in the proof of Theorem \ref{thm:LR_bound_max_error_Gelfand} closely resembles the inequality \eqref{eq:decomp1_bound_e_kplus1_decomposition_using_exact_flow_map}, which we used to prove Theorem \ref{thm:LR_bound_max_error_Banach}. The key difference results from adding $0=\psi(h_k,t_k,u(t_k))-\psi(h_k,t_k,u(t_k))$ before applying the triangle inequality to derive \eqref{eq:decomp1_bound_e_kplus1_decomposition_using_numerical_flow_map}; for \eqref{eq:decomp1_bound_e_kplus1_decomposition_using_exact_flow_map}, we added $0=\varphi(h_k,t_k,U_k)-\varphi(h_k,t_k,U_k)$ instead.
 The decomposition we use for \eqref{eq:decomp1_bound_e_kplus1_decomposition_using_numerical_flow_map} enables us to exploit the weaker local truncation error bound \eqref{eq:local_truncation_error_gelfand} in Assumption \ref{asmp:local_truncation_error_Lipschitz_numerical_flow_Gelfand} instead of the uniform local truncation error bound in Assumption \ref{asmp:Uniform_truncation_error_Banach}.
\end{remark}

\section{Example: Heat equation}
\label{sec:example}

Consider the heat equation on a $C^2$ bounded domain $D\subset\bb{R}^{d}$ with homogeneous Dirichlet boundary conditions
\begin{equation}
 \label{eq:heat_equation}
    u(0) = u_0, \quad \partial_t u - \text{div} (\mathcal{E} \nabla u ) = b \, \text{ on } [0,T] \times D,
\end{equation}
where $\mathcal{E}$ is elliptic. 
Upon multiplying the PDE by a test function and using integration by parts, the left-hand side of the PDE yields a bilinear form $a(u(t),v)$, which allows us to rewrite the problem above as the operator differential equation
\begin{equation}
 \label{eq:parabolic_evolution_equation}
    u(0)=u_0\in H,  \quad u'(t) + A u(t) = b \in V'\, 
\end{equation}
with spaces $H=L^2(D)$, $V=H^1_0(D)$, and $V'=H^{-1}(D)$. 
The bounded, linear operator $A\, : \, V \rightarrow V'$ is induced by the bilinear form $a(\cdot,\cdot)$ on $V\times V$ according to $a(u,v) = \ang{ Au, v }_{V'\times V}$, where $\ang{\cdot,\cdot}_{V'\times V}$ denotes the dual pairing.
For the particular PDE considered above, the operator $A$ is strongly positive with constant $\mu>0$ on $V\times V$. 

In this section, we will show that the results that we proved for the variational setting in Section \ref{sec:Gelfand} are valid for parabolic PDEs and the implicit Euler method, by showing that the conditions \eqref{eq:local_truncation_error_gelfand} and \eqref{eq:global_Lipschitz_approx_flow_gelfand} from Assumption \ref{asmp:local_truncation_error_Lipschitz_numerical_flow_Gelfand} are satisfied.
We shall consider the more general setting of parabolic PDEs with possibly time-dependent coefficients, because this analysis includes the setting of time-independent coefficients --- and hence the heat equation stated above --- as a special case.

Let $L(V,V')$ be the set of all linear mappings from $V$ to $V'$. 
Consider a mapping $a \colon [0,T] \times V \times V \to \mathbb{R}$ that is bilinear in the second and third argument.
This mapping induces a collection $(A(t))_t\subset L(V,V')$ according to 
\begin{equation*}
 \ang{ A(t) u(t), v }_{V'\times V} = a(t,u(t),v),\quad \forall v \in V.
\end{equation*}
Now we pose the following standard assumptions on $a$ and state their equivalent formulation in terms of $A$.
\begin{assumption}
\label{asmp:section_example_bilinear_form}
\begin{enumerate}
    \item For fixed $t$, $a(t,\cdot , \cdot)$ is a bilinear form, and for fixed $u,v\in V$, $a(\cdot, u,v)$ is measurable. Equivalently, for every $t$, $A(t)\in L(V,V')$ is linear and $t\mapsto A(t)$ is measurable.
    \item There exists $\beta>0$ such that for every $(t,u,v)$, $a(t,u,v) \leq \beta \Abs{ u }_V \abs{ v}_V$.
    Equivalently, for every $t$ we have $\Vrt{ A(t)}_{L(V,V')} \leq \beta$.
    \item A G{\aa}rding inequality holds, i.e.\ there exist $\mu>0$, $\kappa\geq 0$ such that 
    \begin{equation}
        \label{eq:positivity}
        a(t,u,u) \geq \mu \Abs{u}_V^2 - \kappa \Abs{ u}_H^2, \quad \forall (t,u)\in [0,T]\times V.
    \end{equation}
    Equivalently, for every $t\in [0,T]$, $A(t) + \kappa I \in L(V,V')$ is strongly positive.
\end{enumerate}
\end{assumption}

For the special case of the heat equation \eqref{eq:heat_equation} where $\mathcal{E}$ is the identity matrix, the first statement of Assumption \ref{asmp:section_example_bilinear_form} holds since $\mathcal{E}$ is constant.
By definition of the bilinear form $a$ and the spaces $H$ and $V$, the second statement holds with $\beta=1$, and the third statement holds with equality for $\kappa=0$ and $\mu=1$.

Consider the implicit Euler scheme
\begin{equation}
   \psi(h,t,v)\coloneqq (I+h \bar{A}_{h,t})^{-1}(h\bar{b}_{h,t}+v),
   \label{eq:implicit_Euler_method_Steklov_averages}
\end{equation}
for $0<h\leq h^\ast$, $0\leq t\leq T-h$ and $v\in H$.
We specify an interval of suitable values of $h^\ast$ in Section \ref{ssec:lipschitz_condition_on_approximate_flow_map}.
Above, $\bar{A}_{h,t}$ and $\bar{b}_{h,t}$ denote Steklov time averages of the linear operators $(A(t))_{t}$ and the right-hand side $b$ respectively,
\begin{equation*}
   \bar{A}_{h,t}\coloneqq \frac{1}{h}\int_{t}^{t+h}A(s)\ud s,\quad \bar{b}_{h,t}\coloneqq \frac{1}{h}\int_{t}^{t+h}b(s)\ud s,
\end{equation*}
where the integrals in the definitions of $\bar{A}_{h,t}$ and $\bar{b}_{h,t}$ are Bochner--Lebesgue integrals in $L(V,V')$ and $V'$ respectively.
The existence of $\psi(h,t,v) \in V$ for $(h\bar{b}_{h,t}+v) \in V'$ is guaranteed by the Lax--Milgram theorem; see e.g.\ \cite[Section 6.2]{Ciarlet2013}.
For every suitable $(h,t)$, the operator $\bar{A}_{h,t}$ inherits the properties of $A$ stated in Assumption \ref{asmp:section_example_bilinear_form}.

For the heat equation \eqref{eq:heat_equation}, $t\mapsto A(t)$ and $t\mapsto b(t)$ are constant.
Therefore, $\bar{A}_{h,t}=A$ and $\bar{b}_{h,t}=b$, and \eqref{eq:implicit_Euler_method_Steklov_averages} simplifies to $\psi(h,t,v)\coloneqq (I+h A)^{-1}(hb+v)$.

\subsection{Local truncation error condition}
\label{ssec:local_truncation_error_condition}

We verify the local truncation error condition \eqref{eq:local_truncation_error_gelfand} in Assumption \ref{asmp:local_truncation_error_Lipschitz_numerical_flow_Gelfand}, for $\psi$ as given in \eqref{eq:implicit_Euler_method_Steklov_averages}.
Recall the definition \eqref{eq:det_exact_seq} of $(u(t_k))_{k}$ and that $(u_k)_{k}$ is defined by $u_0=\vartheta$, $u_{k+1} \coloneqq \psi(h_k, t_k,u_k)$ for $k\in[N-1]_0$.
Under the assumption that $(b-u')' \in L^2(0,T;V')$, the result \cite[Satz 8.3.6]{Emmrich2004} yields for any initial condition $\vartheta\in H$
\begin{equation*}
    \Abs{ u_k - u(t_k) }_H^2 + \mu \sum_{j=1}^k h_j \Abs{ u_j - u(t_j) }_V^2 \leq \frac{h^2}{3 \mu} \Abs{ (b-u')'}_{L^2(0,T;V')}^2,
\end{equation*}
where $\mu$ is the constant from positivity assumption on $A$ \eqref{eq:positivity}. 
Thus, \eqref{eq:local_truncation_error_gelfand} holds with $q=0$ and $C_{\varphi,\psi}(t,x)=(3\mu)^{-1/2}\abs{ (b-u')'}_{L^2(0,T;V')}$ for all $(t,x)$.

One can obtain numerical methods of higher order $q$, by assuming higher regularity of the solution.
For example, \cite[Theorems 4.2, 4.3, 4.4]{Lubich1995A} assume $u,u',u'' \in \mathcal{W}^2(0,T) $, and show the existence of a numerical method $\psi$ that satisfies \eqref{eq:local_truncation_error_gelfand} with $q=1$. 
For a general result dealing with arbitrary regularity $u^{(k+1)} \in \mathcal{W}^2(0,T)$ and numerical method of order $q=k$, see \cite[Theorem 3.2]{Lubich1995B}.

\subsection{Lipschitz condition on approximate flow map}
\label{ssec:lipschitz_condition_on_approximate_flow_map}

Next, we verify the Lipschitz condition \eqref{eq:global_Lipschitz_approx_flow_gelfand} for $\psi$ given in \eqref{eq:implicit_Euler_method_Steklov_averages}, and determine an interval of suitable values for the upper bound $h^\ast$ on the time step of the implicit Euler scheme.
Fix $0<h\leq h^\ast$, $t\in[0,T-h]$, and $u_0,v_0\in V$.
Test $w_1\coloneqq \psi(h,t,u_0)-\psi(h,t,v_0)\in V$ with $w_0\coloneqq u_0-v_0\in H$. 
Then
\begin{align*}
    \frac{1}{2h} (\Abs{w_1}_H^2 -\Abs{w_0}_H^2) \leq & \Ang{\frac{w_1-w_0}{h} , w_1}_{H} \leq - \Ang{\bar{A}_{h,t} w_1 , w_1}_{V'\times V}
    \\
    \leq& - \mu \Abs{ w_1}_V^2 + \kappa \Abs{w_1}_H^2.
\end{align*}
The first inequality follows from rearranging $0\leq \abs{w_{1}+w_{0}}_{H}^{2}$. 
The second inequality follows since  \eqref{eq:implicit_Euler_method_Steklov_averages} is equivalent to $h^{-1}(\psi(h,t,u_0)-u_0)=\bar{b}_{h,t}-\bar{A}_{h,t} \psi(h,t,u_0)$.
The third inequality holds because $\bar{A}_{h,t}$ inherits the positivity property \eqref{eq:positivity} from $A$.
Using $(2h)^{-1}(\abs{w_1}_{H}^{2}-\abs{w_0}_{H}^{2})\leq -\mu\abs{w_1}_{V}^{2}+\kappa\abs{w_1}_{H}^{2}$ and the definitions of $w_1$ and $w_0$, we obtain
\begin{equation*}
     \Abs{\psi(h,t,u_0)-\psi(h,t,v_0)}_H^2 \leq (1-2h\kappa)^{-1}\Abs{u_0-v_0}_H^2.
\end{equation*}
If $\kappa\leq 0$, then $ (1-2h\kappa)^{-1}\leq (1+L_\psi h)$ for any $L_\psi>0$ and $h>0$.
Therefore, suppose that $\kappa>0$. 
If the bound above on $\Abs{\psi(h,t,u_0)-\psi(h,t,v_0)}_H^2$ holds for all $0<h\leq h^\ast$, then we must have $h^\ast<(2\kappa)^{-1}$. 
In fact, if $h^\ast<(2\kappa)^{-1}$, then $L_\psi\coloneqq [(2\kappa)^{-1}-h^\ast]^{-1}$ is equivalent to $h^\ast=\tfrac{L_\psi-2\kappa}{2\kappa L_\psi}$. In this case, $0<h\leq h^\ast$ is equivalent to
\begin{equation*}
      h^2(2\kappa L_\psi)\leq h(L_\psi-2\kappa)\Leftrightarrow 1\leq (1+L_\psi h)(1-2h\kappa)\Leftrightarrow (1-2h\kappa)^{-1} \leq 1+L_\psi h.
\end{equation*}
Hence, the implicit Euler scheme \eqref{eq:implicit_Euler_method_Steklov_averages} satisfies condition \eqref{eq:global_Lipschitz_approx_flow_gelfand} in Assumption \ref{asmp:local_truncation_error_Lipschitz_numerical_flow_Gelfand}. 

\section{Conclusion}
\label{sec:conclusion}

In this paper, we proved strong error bounds for general Orlicz norms for randomised time integration methods applied to operator differential equations, using possibly non-uniform time grids. 
Our work builds on the ideas and approaches of \cite{Conrad2017,LieStuartSullivan2019}.
We show that the proof techniques of the key error bounds contained therein can be applied in more general settings, where the differential equation is formulated on a possibly infinite-dimensional Banach or Hilbert space, and the numerical time integration method is applied to a possibly non-uniform time grid.
Our work has two additional novel aspects relative to \cite{Conrad2017,LieStuartSullivan2019}.

First, we use a different error decomposition to bound the one-step error.
Our error decomposition enables us to replace the strong assumption of uniform local truncation error with a weaker assumption on the local truncation error.
This is important, because it is known that the strong assumption of uniform local truncation error is invalid even when the linear operator $A$ in the operator differential equation generates an analytic semigroup \cite[Theorem 7.1]{Thomee2006}. 
For the implicit Euler method, and for a large class of examples that includes the standard heat equation, we showed that our weaker local truncation error assumption is reasonable.

Second, we consider more general Orlicz norms instead of $L^R$-norms.
Previous results concerning higher-order error bounds --- for example, \cite[Theorem 3.5]{LieStuartSullivan2019} --- were less direct: they involved finding bounds on the $L^R$ error for each $R\in\bb{N}$ and using the series expansion of the exponential function.
The use of Orlicz norms leads to shorter and conceptually simpler proofs of our main results, Theorem \ref{thm:LR_bound_max_error_Banach} and Theorem \ref{thm:LR_bound_max_error_Gelfand}, by exploiting the fact that the random approximating sequence $(U_k)_k$ inherits the integrability properties of the collection $(\xi_k)_k$.  

\section*{Acknowledgements}

The research of HCL and MS has been partially funded by the Deutsche Forschungsgemeinschaft (DFG) --- Project-ID 318763901 --- SFB1294. The authors thank the anonymous reviewer for their constructive and helpful feedback.

\appendix

\section{Taylor expansion in Banach Spaces}\label{app:Taylor}
The following version of Taylor's theorem in Banach spaces is given in \cite[Theorem 7.9-1]{Ciarlet2013}.
\begin{theorem}
    Let $V$ and $W$ be normed vector spaces, let $U$ be an open subset of $V$, let $[a,a+h]$ be a closed segment contained in $U$, let $f \colon U \rightarrow W$ be a given mapping, and let $m\in \mathbb{N}$.
    \begin{enumerate}[(a)]
        \item (Taylor-Young) If $f$ is $(m-1)$ times differentiable on $U$ and $m$ times differentiable at $a \in U$, then
        \begin{equation*}
            f(a+h) = f(a) + f'(a) h + \ldots + \frac{1}{m!} f^{(m)} (a) h^m + \Vrt{ h }_V^m \delta (h)
        \end{equation*}
        with $\lim_{h\rightarrow 0} \delta (h) = 0$.
        \item (Integral remainder) If $W$ is a Banach space and $f$ is $m$ times continuously differentiable on $U$, then
        \begin{align*}
            f(a+h) &= f(a) + f'(a) h + \ldots + \frac{1}{(m-1)!} f^{(m-1)} (a) h^{m-1}\\
            & \phantom{=} \quad + \frac{1}{(m-1)!} \int_0^1 (1-t)^{m-1} \left( f^{(m)}(a+th) h^m \right) \,\mathrm{d}t.
        \end{align*}
    \end{enumerate}
\end{theorem}
The noteworthy differences to the standard Taylor theorem in $\mathbb{R}$ are: 1) differentiability of $f$ may be slightly more complicated; 2) the $k$-th derivative of $f$ is a $k$-linear mapping from $\Pi_{i=1}^k U$ to $W$, taking $k$ inputs from $U$, denoted by $h^k = (h,\ldots , h) \in U^k$.

In addition, one can derive a Taylor expansion almost everywhere for weakly differentiable functions.

\section{Additional material for Section \ref{sec:Banach}}
\label{app:one-step}

 In the setting of time integration for ODEs, one usually requires higher regularity of $f$ or equivalently higher regularity for the solution $u$ in order to achieve an order of $q\geq 1$ for the truncation error \cite[Section III.2, Theorem 2.4]{Hairer1993}. 
 The purpose of this section is to show that the same ideas apply in the infinite-dimensional Banach space setting.
 Below, the function $f$ refers to the vector field in \eqref{eq:operator_differential_equation}.
 
For the next lemma, we consider for a general explicit Euler one-step method the map $\psi \colon  [0,h]\times \mathbb{R} \times V \rightarrow V$ in a Banach space $(V,\Abs{\cdot}_V)$, associated to some step function $\Upsilon$:
\begin{align}\label{eq:one_step_Upsilon}
	u_{k+1} &\coloneqq \psi(h_k,t_k,u_k)\coloneqq u_k + h_k \Upsilon(h_k,t_k,u_k) \\
	&= u_k +h_k  \left( a_1 f(t_k,u_k)  + a_2f\left(t_k+b_1 h_k,t_k,u_k + b_2 h_k f(t_k,u_k) \right) \right)\nonumber
\end{align}
with $a_1,a_2,b_1,b_2 \geq 0$.
\begin{lemma}[Lipschitz property of the numerical flow map]
	Let $\psi$ be as in \eqref{eq:one_step_Upsilon}.
	If $f \colon [0,T] \times V \rightarrow V$ is Lipschitz continuous in the second argument, then the approximate flow map $\psi(h_k,t_k,\cdot)$ is Lipschitz continuous, uniformly in $k$.
\end{lemma}
\begin{proof}
    We use the Lipschitz property of $f$ multiple times to get
	\begin{equation}
		\Abs{ \psi(\tau,t,u) - \psi(\tau,t,v) }_V \leq (1+ L \tau (a_1 L + a_2 L + a_2b_2 L^2 \tau)) \Abs{ u - v}_V \; .
	\end{equation}
\end{proof}
The following theorem is taken from \cite[Theorem 7.1.5]{Plato2010}. The proof there is only done in one dimension.
\begin{theorem}
	Assuming that the right hand side $f$ belongs to $C^2([0,T]\times V;V)$ and equivalently that the solution $u$ is $C^3$, it follows that any explicit one-step scheme with step function $\Upsilon$ as in \eqref{eq:one_step_Upsilon} with
	\begin{equation*}
	    a_1+a_2 = 1 \, , \, a_2 b_1 = \frac{1}{2} \, , \, a_2 b_2 = \frac{1}{2}
	\end{equation*}
	satisfies assumption \ref{asmp:Uniform_truncation_error_Banach} with $q = 2$.
\end{theorem}
\begin{proof}
	We assume that $f \in C^2([0,T]\times V;V)$ in order to be able to use Taylor expansion. We first expand the step function $\Upsilon$:
	\begin{equation*}
	    \Upsilon (h,t,u(t)) = \left[ (a_1+a_2) f(t,u(t)) + h \left( a_2 b_1  \frac{\partial f}{\partial t} (t,u(t)) + a_2 b_2  \frac{\partial f}{\partial u }  f(t,u(t))  \right) \right] + \mathcal{O} (h^2).
	\end{equation*}
	Above, $\tfrac{\partial f}{\partial u}$ refers to the Gateaux derivative of $f$ with respect to $u$, and $\tfrac{\partial f}{\partial u }  f (t,u(t))$ denotes the linear mapping $\tfrac{\partial  f}{\partial u} (t,u(t))$ from $V$ to $V$ acting on $f(t,u(t)) \in V$. 
	Expanding $u$, we have
	\begin{align*}
	    u(t+h) &= u(t) + u'(t) h + u''(t) \frac{h^2}{2} + \mathcal{O}(h^3)\\
	    &= u(t) + \left[ h f(t,u(t)) + \frac{h^2}{2} \left( \frac{\partial f}{\partial t} (t, u(t)) + \frac{\partial f}{\partial u} (t,u(t)) f(t,u(t)) \right) \right]+ \mathcal{O} (h^3)\\
	    &= u(t) +\left[ h \Upsilon (h,t,u(t)) + \mathcal{O}(h^3) \right] + \mathcal{O}(h^3) \; .
	\end{align*}
	The last equality follows from the conditions on the coefficients $a_1,b_1,a_2,b_2$. This gives consistency of order $2$:
	\begin{equation*}
	    u (t+h) - (u (t) + h \Upsilon (h,t,u(t)) ) = \varphi(h,t,u(t)) - \psi (h,t,u(t)) =  \mathcal{O}(h^3).
	\end{equation*}
	Convergence follows from the Lipschitz continuity of $f$ with respect to the second argument \cite[Theorem 7.10]{Plato2010}. 
\end{proof}

\section{Discrete Gronwall inequalities}

The following statement is given in \cite[Lemma 1.6]{MilsteinTretyakov2004}.
\begin{lemma}
\label{lem:MilsteinTretyakov_GronwallLemma}
 Let $T>0$ be fixed. Let $N\in\bb{N}$ and $h=T/N$. Suppose $(y_k)_k\in[0,\infty)^{\bb{N}_0}$ is such that for some $A,B\geq 0$ and $p\geq 1$,
 \begin{equation*}
  y_{k+1}\leq (1+Ah)y_k+Bh^p,\quad k\in[N-1]_0.
 \end{equation*}
 Then
 \begin{equation*}
  y_k\leq e^{AT}y_0+\tfrac{B}{A}(e^{AT}-1)h^{p-1},\quad k\in [N-1]_0
 \end{equation*}
where for $A=0$, $A^{-1}(e^{AT}-1)=0$.
\end{lemma}

We restate the ``special Gronwall inequality'' of \cite{Holte2009}.
\begin{proposition}
 \label{prop:Special_Gronwall_Inequality}
 Let $(y_n)_n,(g_n)_n\in[0,\infty)^{\bb{N}_0}$, $c\geq 0$, and $N\in\bb{N}$ be arbitrary. If 
 \begin{equation*}
  y_{k+1}\leq c+\sum_{0\leq j\leq k}g_j y_j,\quad k\in[N-1]_0
 \end{equation*}
 then 
 \begin{equation*}
  y_{k+1}\leq c\exp\left(\sum_{0\leq j\leq k}g_j\right),\quad k\in[N-1]_0.
 \end{equation*}
\end{proposition}
The following lemma is a corollary of Proposition \ref{prop:Special_Gronwall_Inequality}.
\begin{lemma}
 \label{lem:Special_Gronwall_Inequality_corollary}
 Let $T>0$ and $N\in\bb{N}$ be fixed. Let $(h_k)_k,(y_k)_k,(b_k)_k\in [0,\infty)^{\bb{N}_0}$ be such that for some $A\geq 0$,
 \begin{equation*}
  y_{k+1}\leq (1+Ah_k)y_k+b_k,\quad k\in[N-1]_0.
 \end{equation*}
 Then 
 \begin{equation*}
  y_{k+1}\leq \left(y_0+\sum_{ \ell\in[N-1]_0}b_\ell\right)\exp\left(\sum_{0\leq j\leq k} Ah_j\right),\quad k\in[N-1]_0.
 \end{equation*}
\end{lemma}
\begin{proof}
 Rewriting the upper bound on $y_{k+1}$ yields
 \begin{equation*}
  y_{j+1}-y_j\leq Ah_j y_j +b_j,\quad j\in [N-1]_0.
 \end{equation*}
 Summing the differences from $j=0$ to $j=k\in[N-1]_0$ yields
 \begin{equation*}
  y_{k+1}\leq y_0+\sum_{0\leq j\leq k} \left(Ah_j y_j+b_j\right)\leq \left(y_0+\sum_{0\leq \ell \leq N-1}b_\ell\right)+\sum_{0\leq j\leq k} Ah_j y_j,\quad k\in[N-1]_0.
 \end{equation*}
Applying Proposition \ref{prop:Special_Gronwall_Inequality} completes the proof.
\end{proof}

\section{Proofs for Section \ref{ssec:L2_error_bounds_independent_centred_case_Gelfand}}
\label{sec:proofs_for_section_L2_error_bounds_independent_centred_case_Gelfand}

\subsection{Proof of Lemma \ref{lem:martingale}}
\label{ssec:proof_of_lemma_martingale}

Lemma \ref{lem:martingale} states that the stochastic process $(M_k)_{k\in[N-1]_0}$ defined by \eqref{eq:Mk_martingale},
  \begin{equation*}
   M_k\coloneqq \sum_{j=0}^{k}\Ang{\varphi(h_j,t_j,u(t_j))-\psi(h_j,t_j,U_j),\xi_j(h_j)}_H,
  \end{equation*}
  is a $\bb{R}$-valued, square-integrable martingale with respect to the filtration $(\mathcal{F}_k)_{k\in[N-1]_0}$ generated by the $(\xi_k(h_k))_{k\in[N-1]_0}$, and that there exists a universal constant $\kappa>0$ such that for every $k\in[N-1]_0$, the bound \eqref{eq:consequence_BDG_inequality_Mk} holds:
  \begin{equation*}
   \bb{E}\left[\max_{j\in[k]_0}\Abs{M_k} \right]\leq   \Vrt{C_{\varphi,\psi}}_{\infty}^2h^{2q+1}+\frac{1}{4}\bb{E}\left[\max_{j\in[k]_0}\abs{e_j}^2_H\right]+\kappa^2(1+L'_\psi) TC_\xi^2 h^{2p+1}.
  \end{equation*}
 \begin{proof}[Proof of Lemma \ref{lem:martingale}]
   For $(M_k)_k$ to satisfy the definition of a martingale, we must show that for every $k$, $M_k\in L^1(\Omega;\bb{R})$ and is $\mathcal{F}_k$-measurable, and that the martingale property $\bb{E}[M_{k+1}-M_{k}\vert \mathcal{F}_k]=0$ holds for $k\in[N-2]_0$. 
  The measurability of $M_k$ with respect to $\mathcal{F}_k$ follows from the definition of $\mathcal{F}_k$ and Lemma \ref{lem:Uk_measurable_wrt_Fk}.
  By the triangle inequality, the Cauchy--Schwarz inequality, \eqref{eq:squared_bound_abs_Phi_utk_minus_Psi_Uk}, and Assumption \ref{asmp:Noise_regularity_Gelfand},
 \begin{align*}
  \Vrt{M_k}_{L^2(\Omega;\bb{R})}\leq & \sum_{j=0}^{k}\Vrt{\varphi(h_j,t_j,u(t_j))-\psi(h_j,t_j,U_j)}_{L^2(\Omega;H)}^2\Vrt{\xi_j(h_j)}_{L^2(\Omega;H)}^2
  \\
  \leq & \sum_{j=0}^{k}\left(3\Vrt{C_{\varphi,\psi}}_{\infty}^2 h_{k}^{2q+1}+(1+L'_\psi h_k)^2\Vrt{e_k}_{L^2(\Omega;H)}^2\right)C_\xi^2 h^{2p+2}.
 \end{align*}
By Lemma \ref{lem:max_L2_error_independent_centred_noise_Gelfand}, $e_k\in L^2(\Omega;H)$ for every $k$, and thus $M_k\in L^2(\Omega;\bb{R})\subset L^1(\Omega;\bb{R})$. Hence, $(M_k)_k$ is a square-integrable martingale.

Next, we prove the martingale property. By Lemma \ref{lem:Uk_measurable_wrt_Fk}, $\xi_{k+1}(h_{k+1})$ is independent of  $U_{k+1}$ and $\mathcal{F}_k$. By the definition \eqref{eq:Mk_martingale} of $M_k$, the tower property of conditional expectation, and the centredness of the $(\xi_k(h_k))_k$,
\begin{align*}
\bb{E}[\ang{\varphi(h_{k+1},u(t_{k+1}))-\psi(h_{k+1},U_{k+1}), \xi_{k+1}(h_{k+1})}_H]=0,
\end{align*}
and thus $(M_k)_k$ is a $(\mathcal{F}_k)_k$-martingale.

Finally, we prove the second statement. Since $(M_k)_k$ is a square integrable martingale, the Burkholder--Davis--Gundy inequality ensures that for every $k\in[N-1]_0$
\begin{equation*}
 \bb{E}\left[\max_{j\in[k]_0}\Abs{M_j} \right]\leq \kappa \bb{E}\left[\Ang{M}_{k}^{1/2}\right]
\end{equation*}
where $\kappa>0$ is the same universal constant appearing in \eqref{eq:consequence_BDG_inequality_Mk}. The quadratic variation process $\ang{M}$ is defined by $\Ang{M}_0\coloneqq 0$ and $\ang{M}_{k}\coloneqq \sum_{j\in[k]}\bb{E}[(M_j-M_{j-1})^2\vert\mathcal{F}_{j-1}]$ for $k\in[N-1]$, see e.g. \cite[Chapter I, Definition 2.3]{RevuzYor2009}.
Using \eqref{eq:Mk_martingale}, the measurability of $U_j$ with respect to $\mathcal{F}_{j-1}$ (cf. Lemma \ref{lem:Uk_measurable_wrt_Fk}), the Cauchy--Schwarz inequality, and \eqref{eq:squared_bound_abs_Phi_utk_minus_Psi_Uk},
\begin{align*}
 \ang{M}_{k}\leq& \sum_{j\in[k]} \bb{E}\left[\Ang{\varphi(h_j,t_j,u(t_j))-\psi(h_j,t_j,U_j),\xi_j(h_j)}^2_H \middle\vert \mathcal{F}_{j-1}\right]
 \\
 \leq& \max_{j\in[k]}\Abs{\varphi(h_j,t_j,u(t_j))-\psi(h_j,t_j,U_j)}_H^2\sum_{j\in[k]} \bb{E}\left[\Abs{\xi_j(h_j)}^2_H \middle\vert \mathcal{F}_{j-1}\right] 
 \\
 \leq& \max_{j\in[k]}\left(3 \Vrt{C_{\varphi,\psi}}_{\infty}^2 h_{j}^{2q+1}+(1+L'_\psi h_j)\Abs{e_j}_H^2\right)\sum_{j\in[k]} \bb{E}\left[\Abs{\xi_j(h_j)}^2_H \middle\vert \mathcal{F}_{j-1}\right].
\end{align*}
The hypothesis that $0<h\leq 1$ was used to obtain \eqref{eq:squared_bound_abs_Phi_utk_minus_Psi_Uk}.
Using Young's inequality with $s,s'>1$ such that $s^{-1}+(s')^{-1}=1$
\begin{equation*}
 ab\leq \frac{\delta}{s} a^{s}+\frac{1}{\delta^{s'/s} s'}b^{s'}
\end{equation*}
with $s=2$ and $\delta=2\kappa(1+L'_\psi h)$, and using $h_k\leq h\leq 1$,
\begin{align*}
 \ang{M}_{k}^{1/2}\leq & \frac{1}{\kappa(4+4L'_\psi h)}\max_{j\in[k]}\left(3 \Vrt{C_{\varphi,\psi}}_{\infty}^2 h_{j}^{2q+1}+(1+L'_\psi h_j)\Abs{e_j}_H^2\right)
 \\
 &+ \kappa(1+L'_\psi h)\sum_{j\in[k]} \bb{E}\left[\Abs{\xi_j(h_j)}^2_H \middle\vert \mathcal{F}_{j-1}\right]
 \\
 \leq & \frac{\Vrt{C_{\varphi,\psi}}_{\infty}^2}{\kappa}h^{2q+1}+\frac{1}{4\kappa}\max_{j\in[k]_0}\abs{e_j}^2_H+\kappa(1+L'_\psi) \sum_{j\in[N-1]_0} \bb{E}\left[\Abs{\xi_j(h_j)}^2_H \middle\vert \mathcal{F}_{j-1}\right].
\end{align*}
By taking expectations, the tower property removes the conditioning on $\mathcal{F}_{j-1}$ in each summand. Using \eqref{eq:upper_bound_on_sum_time_steps_to_power_plus_one} and the Burkholder--Davis--Gundy inequality completes the proof.
\end{proof}

\subsection{Proof of Proposition \ref{prop:L2_max_error_bound_independent_centred_noise_Gelfand}}
\label{ssec:proof_of_proposition_L2_max_error}

Proposition \ref{prop:L2_max_error_bound_independent_centred_noise_Gelfand} states the error bound
\begin{align*}
 &\Vrt{\max_{k\in[N]_0}\Abs{e_k}_H}_{2}^2
 \\
  \leq & 2\left(\Vrt{e_0}^2_{2}+4\Vrt{C_{\varphi,\psi}}_{\infty}^2 h^{2q}T+ C_\xi^2 Th^{2p+1} (1+\kappa^2(1+L'_\psi))\right)\exp\left(2L'_\psi T\right)
\end{align*}
for the same universal constant $\kappa$ in \eqref{eq:consequence_BDG_inequality_Mk}.
 \begin{proof}[Proof of Proposition \ref{prop:L2_max_error_bound_independent_centred_noise_Gelfand}]
  Since $0<h\leq 1$, we may use \eqref{eq:bound_abs_ekplus1_squared} to obtain
   \begin{align*}
   \abs{e_{k+1}}^2_H-\abs{e_k}_H^2 \leq& L'_\psi h_k \abs{e_k}_H^2+3 \Vrt{C_{\varphi,\psi}}_{\infty}^2 h_{k}^{2q+1}+\Abs{\xi_k(h_k)}_H^2
  \\
  &+2\Ang{\varphi(h_k,t_k,u(t_k))-\psi(h_k,t_k,U_k), \xi_k(h_k)}_H.
  \end{align*}
   Using that $\sum_{j\in[k+1]_0}(\abs{e_{k+1}}^2_H-\abs{e_k}_H^2)=\abs{e_{k+1}}_H^2-\abs{e_0}_H^2$, and using the definition \eqref{eq:Mk_martingale} of the martingale $(M_k)_k$, we obtain
 \begin{align*}
  \abs{e_{k+1}}_H^2\leq &\abs{e_0}_H^2+\sum_{j\in[N-1]_0}\left(3 \Vrt{C_{\varphi,\psi}}_{\infty}^2 h_{j}^{2q+1}+ \Abs{\xi_j(h_j)}_H^2\right)
  \\
  &+2 M_k+L'_\psi \sum_{j\in[k]_0}h_j\abs{e_j}_H^2.
 \end{align*}
 Since only $M_k$ can attain negative values, the above bound implies
 \begin{align*}
  \max_{j\in[k+1]_0}\abs{e_j}_H^2\leq &\abs{e_0}_H^2+\sum_{j\in[N-1]_0}\left(3 \Vrt{C_{\varphi,\psi}}_{\infty}^2 h_{j}^{2q+1}+ \Abs{\xi_j(h_j)}_H^2\right)
  \\
  &+2\max_{j\in [k]_0}\abs{ M_j}+L'_\psi \sum_{j\in[k]_0}h_j\max_{\ell\in[j]_0}\abs{e_\ell}_H^2.
 \end{align*}
 Take expectations, apply Assumption \ref{asmp:Noise_regularity_Gelfand}, apply the bound \eqref{eq:consequence_BDG_inequality_Mk} from Lemma \ref{lem:martingale}, and use that $\bb{E}[\max_{j\in [k]_0}\abs{e_j}_H^2] \leq \bb{E}[\max_{j\in [k+1]_0}\abs{e_j}_H^2]$ to obtain
 \begin{align*}
  \bb{E}\left[\max_{j\in[k+1]_0}\abs{e_j}_H^2\right]\leq & \bb{E}\left[\abs{e_0}_H^2\right]+3 \Vrt{C_{\varphi,\psi}}_{\infty}^2 T h^{2q}+ C_\xi^2 Th^{2p+1}
  \\
  &+2\bb{E}\left[\max_{j\in [k]_0}\abs{ M_j}\right]+L'_\psi \sum_{j\in[k]_0}h_j\bb{E}\left[\max_{\ell\in[j]_0}\abs{e_\ell}_H^2\right]
  \\
  \leq & \bb{E}\left[\abs{e_0}_H^2\right]+4\Vrt{C_{\varphi,\psi}}_{\infty}^2 h^{2q}T+ C_\xi^2 Th^{2p+1}(1+\kappa^2(1+L'_\psi))
  \\
  &+\frac{1}{2}\bb{E}\left[\max_{j\in [k+1]_0}\abs{e_j}_H^2\right]+L'_\psi \sum_{j\in[k]_0}h_j\bb{E}\left[\max_{\ell\in[j]_0}\abs{e_\ell}_H^2\right].
 \end{align*}
 Subtracting $\tfrac{1}{2}\bb{E}[\max_{j\in [k+1]_0}\abs{e_j}_H^2]$ from both sides and applying the discrete Gronwall inequality in Proposition \ref{prop:Special_Gronwall_Inequality} completes the proof.
 \end{proof}

\bibliographystyle{amsplain} 
\bibliography{bibliofile}

\providecommand{\bysame}{\leavevmode\hbox to3em{\hrulefill}\thinspace}
\providecommand{\MR}{\relax\ifhmode\unskip\space\fi MR }
\providecommand{\MRhref}[2]{%
  \href{http://www.ams.org/mathscinet-getitem?mr=#1}{#2}
}
\providecommand{\href}[2]{#2}
\begin{thebibliography}{10}

\bibitem{AbdulleGaregnani2020}
Assyr {Abdulle} and Giacomo {Garegnani}, \emph{{Random time step probabilistic
  methods for uncertainty quantification in chaotic and geometric numerical
  integration}}, {Stat. Comput.} \textbf{30} (2020), no.~4, 907--932.

\bibitem{AbdulleGaregnani2021}
\bysame, \emph{A probabilistic finite element method based on random meshes:
  {E}rror estimators and {B}ayesian inverse problems}, 2021, arXiv:2103.06204.

\bibitem{Adams2004}
Robert~A. Adams and John J.~F. Fournier, \emph{Sobolev spaces}, second ed.,
  Pure and Applied Mathematics (Amsterdam), vol. 140, Elsevier/Academic Press,
  Amsterdam, 2003.

\bibitem{Bogachev1998Gaussian}
Vladimir~I. Bogachev, \emph{Gaussian {Measures}}, Mathematical Surveys and
  Monographs, vol.~62, American Mathematical Society, Providence, RI, 1998.

\bibitem{Chkrebtii2019}
Oksana~A. {Chkrebtii} and David~A. {Campbell}, \emph{{Adaptive step-size
  selection for state-space probabilistic differential equation solvers}},
  {Stat. Comput.} \textbf{29} (2019), no.~6, 1285--1295.

\bibitem{Chkrebtii2016}
Oksana~A. Chkrebtii, David~A. Campbell, Ben Calderhead, and Mark~A. Girolami,
  \emph{Bayesian solution uncertainty quantification for differential
  equations}, Bayesian Anal. \textbf{11} (2016), no.~4, 1239--1267.

\bibitem{Ciarlet2013}
Philippe~G. {Ciarlet}, \emph{{Linear and Nonlinear Functional Analysis with
  Applications}}, vol. 130, Society for Industrial and Applied Mathematics,
  Philadelphia, PA, 2013.

\bibitem{Cockayne2017a}
Jon Cockayne, Chris Oates, T.~J. Sullivan, and Mark Girolami,
  \emph{Probabilistic numerical methods for {PDE}-constrained {Bayesian}
  inverse problems}, Proceedings of the 36\textsuperscript{th} {International}
  {Workshop} on {Bayesian} {Inference} and {Maximum} {Entropy} {Methods} in
  {Science} and {Engineering} (Geert Verdoolaege, ed.), AIP Conference
  Proceedings, vol. 1853, 2017, pp.~060001--1--060001--8.

\bibitem{Cockayne2019}
Jon Cockayne, Chris~J. Oates, T.~J. Sullivan, and Mark Girolami, \emph{Bayesian
  probabilistic numerical methods}, SIAM Rev. \textbf{61} (2019), no.~4,
  756--789.

\bibitem{Conrad2017}
Patrick~R. Conrad, Mark Girolami, Simo S{\"a}rkk{\"a}, Andrew~M. Stuart, and
  Konstantinos~C. Zygalakis, \emph{Statistical analysis of differential
  equations: introducing probability measures on numerical solutions}, Stat.
  Comput. \textbf{27} (2017), no.~4, 1065--1082.

\bibitem{DupontEnsslin2018a}
Martin Dupont and Torsten En\ss{}lin, \emph{Consistency and convergence of
  simulation schemes in information field dynamics}, Phys. Rev. E \textbf{98}
  (2018), 043307.

\bibitem{Emmrich2004}
Etienne Emmrich, \emph{{Gew\"ohnliche und
  Operator-Differential\-glei\-chun\-gen. Eine integrierte Einf\"uhrung in
  Randwertprobleme und Evolutionsgleichungen f\"ur Studierende.}}, Wiesbaden:
  Vieweg, 2004 (German).

\bibitem{Emmrich2007}
Etienne Emmrich and Olaf Weckner, \emph{Analysis and numerical approximation of
  an integro-differential equation modeling non-local effects in linear
  elasticity}, Math. Mech. Solids \textbf{12} (2007), no.~4, 363--384.

\bibitem{Ensslin2013}
Torsten~A. En\ss{}lin, \emph{Information field dynamics for simulation scheme
  construction}, Phys. Rev. E \textbf{87} (2013), 013308.

\bibitem{Garegnani2021}
Giacomo {Garegnani}, \emph{Sampling methods for {B}ayesian inference involving
  convergent noisy approximations of forward maps}, 2021, arXiv:2111.03491.

\bibitem{Hairer1993}
E.~Hairer, S.~P. N{\o}rsett, and G.~Wanner, \emph{Solving ordinary differential
  equations. {I}}, second ed., Springer Series in Computational Mathematics,
  vol.~8, Springer-Verlag, Berlin, 1993, Nonstiff problems.

\bibitem{HennigOsborneGirolami2015}
Philipp Hennig, Michael~A. Osborne, and Mark Girolami, \emph{Probabilistic
  numerics and uncertainty in computations}, P. Roy. Soc. Lond. A Mat.
  \textbf{471} (2015), no.~2179, 20150142.

\bibitem{Holte2009}
John~M. Holte, \emph{Discrete {Gronwall} lemma and applications}, 2009,
  \url{http://homepages.gac.edu/~holte/publications/GronwallLemma.pdf}.
  Accessed 14-12-2021.

\bibitem{Kersting2020}
Hans {Kersting}, T.~J. {Sullivan}, and Philipp {Hennig}, \emph{{Convergence
  rates of Gaussian ODE filters}}, {Stat. Comput.} \textbf{30} (2020), no.~6,
  1791--1816.

\bibitem{LieStuartSullivan2019}
Han~Cheng Lie, Andrew~M. Stuart, and T.~J. Sullivan, \emph{Strong convergence
  rates of probabilistic integrators for ordinary differential equations},
  Stat. Comput. \textbf{29} (2019), no.~6, 1265--1283.

\bibitem{LieSullivanTeckentrup2018}
Han~Cheng Lie, T.~J. Sullivan, and Aretha~L. Teckentrup, \emph{Random forward
  models and log-likelihoods in {B}ayesian inverse problems}, SIAM/ASA J.
  Uncertain. Quantif. \textbf{6} (2018), no.~4, 1600--1629.

\bibitem{Lubich1995A}
Christian Lubich and Alexander Ostermann, \emph{Linearly implicit time
  discretization of non-linear parabolic equations}, IMA J. Numer. Anal.
  \textbf{15} (1995), no.~4, 555--583.

\bibitem{Lubich1995B}
\bysame, \emph{Runge-{K}utta approximation of quasi-linear parabolic
  equations}, Math. Comp. \textbf{64} (1995), no.~210, 601--627.

\bibitem{Matsuda2019}
Takeru Matsuda and Yuto Miyatake, \emph{Estimation of ordinary differential
  equation models with discretization error quantification}, SIAM/ASA J.
  Uncertain. Quantif. \textbf{9} (2021), no.~1, 302--331.

\bibitem{Meinlschmidt2020}
Hannes Meinlschmidt, Christian Meyer, and Stephan Walther, \emph{{Optimal
  control of an abstract evolution variational inequality with application to
  homogenized plasticity}}, {J. Nonsmooth Anal. Optim.} \textbf{1} (2020),
  1--41.

\bibitem{MilsteinTretyakov2004}
G.~N. Milstein and M.~V. Tretyakov, \emph{Stochastic numerics for mathematical
  physics}, Scientific Computation, Springer-Verlag, Berlin, 2004.

\bibitem{Oates2019}
C.~J. Oates and T.~J. Sullivan, \emph{A modern retrospective on probabilistic
  numerics}, Stat. Comput. \textbf{29} (2019), no.~6, 1335--1351.

\bibitem{Owhadi2015}
Houman Owhadi, \emph{Bayesian numerical homogenization}, Multiscale Model.
  Simul. \textbf{13} (2015), no.~3, 812--828.

\bibitem{Owhadi2017}
Houman {Owhadi} and Lei {Zhang}, \emph{{Gamblets for opening the
  complexity-bottleneck of implicit schemes for hyperbolic and parabolic
  ODEs/PDEs with rough coefficients}}, {J. Comput. Phys.} \textbf{347} (2017),
  99--128.

\bibitem{Plato2010}
Robert {Plato}, \emph{{Numerische Mathematik kompakt. Grundlagenwissen f\"ur
  Studium und Praxis.}}, Wiesbaden: Vieweg, 2004 (German).

\bibitem{Raissi2018}
Maziar Raissi, Paris Perdikaris, and George~Em Karniadakis, \emph{Numerical
  {Gaussian} processes for time-dependent and nonlinear partial differential
  equations}, SIAM J. Sci. Comput. \textbf{40} (2018), no.~1, A172--A198.

\bibitem{RevuzYor2009}
Daniel Revuz and Marc Yor, \emph{Continuous martingales and brownian motion},
  third ed., Grundlehren der mathematischen Wissenschaften, vol. 293,
  Springer-Verlag, Berlin, 2009, Corrected Third Printing.

\bibitem{Schober2019}
Michael {Schober}, Simo {S\"arkk\"a}, and Philipp {Hennig}, \emph{{A
  probabilistic model for the numerical solution of initial value problems}},
  {Stat. Comput.} \textbf{29} (2019), no.~1, 99--122.

\bibitem{Stuart2010}
Andrew~M. Stuart, \emph{Inverse problems: a {B}ayesian perspective}, Acta
  Numer. \textbf{19} (2010), 451--559.

\bibitem{Teymur2018}
Onur Teymur, Han~Cheng Lie, T.~J. Sullivan, and Ben Calderhead, \emph{Implicit
  probabilistic integrators for {ODE}s}, Advances in Neural Information
  Processing Systems 31 (NIPS 2018) (S.~Bengio, H.~Wallach, H.~Larochelle,
  K.~Grauman, N.~Cesa-Bianchi, and R.~Garnett, eds.), vol.~31, Curran
  Associates, Inc., 2018.

\bibitem{Teymur2016}
Onur Teymur, Konstantinos Zygalakis, and Ben Calderhead, \emph{{P}robabilistic
  {L}inear {M}ultistep {M}ethods}, Advances in Neural Information Processing
  Systems 29 (D.~D. Lee, M.~Sugiyama, U.~V. Luxburg, I.~Guyon, and R.~Garnett,
  eds.), Curran Associates, Inc., 2016, pp.~4321--4328.

\bibitem{Thomee2006}
Vidar {Thom\'ee}, \emph{{Galerkin finite element methods for parabolic
  problems.}}, Berlin: Springer, 2006.

\bibitem{Tronarp2019}
Filip {Tronarp}, Hans {Kersting}, Simo {S\"arkk\"a}, and Philipp {Hennig},
  \emph{{Probabilistic solutions to ordinary differential equations as
  nonlinear Bayesian filtering: a new perspective}}, {Stat. Comput.}
  \textbf{29} (2019), no.~6, 1297--1315.

\bibitem{Wang2021}
Junyang Wang, Jon Cockayne, Oksana Chkrebtii, T.~J. Sullivan, and Chris Oates,
  \emph{Bayesian numerical methods for nonlinear partial differential
  equations}, Stat. Comput. \textbf{31} (2021), no.~5, no. 55, 20pp.

\bibitem{Wang2020}
Junyang Wang, Jon Cockayne, and Chris Oates, \emph{A role for symmetry in the
  {B}ayesian solution of differential equations}, Bayesian Anal. \textbf{15}
  (2020), no.~4, 1057--1085.

\bibitem{ZeidlerIIA}
Eberhard Zeidler, \emph{{Nonlinear Functional Analysis and its Applications.
  {II}/{A}}}, Springer-Verlag, New York, 1990, Linear monotone operators,
  Translated from the German by the author and Leo F. Boron.

\end{thebibliography}

\end{document}